\documentclass[11pt]{article}
\usepackage[utf8]{inputenc} % allow utf-8 input
\usepackage[T1]{fontenc}    % use 8bit T1 fonts
\usepackage{titlesec}
\usepackage{amssymb}
\usepackage{amsthm}
\usepackage{mathtools}
\usepackage{nicefrac}       % compact symbols for 1/2, etc.
\usepackage{bbm}
\usepackage[scr=boondoxo]{mathalpha}

\usepackage{graphicx}
\usepackage{tikz}
\usetikzlibrary{decorations.pathreplacing}
\usepackage{tikz-cd}
\usepackage[all, cmtip]{xy}

\usepackage{booktabs}       % professional-quality tables
\usepackage{caption}
\usepackage{microtype}      % microtypography
\usepackage{enumitem}
\usepackage{todonotes}
\usepackage{verbatim}
\usepackage{soul}

\usepackage[noadjust]{cite}
\usepackage[
    colorlinks=true,
    linkcolor=blue,
    citecolor=blue,
    urlcolor=blue,
    breaklinks=true
    ]{hyperref}
\usepackage[capitalise]{cleveref}

\usepackage[
    top=2.9cm,
    bottom=2.9cm,
    left=2.5cm,
    right=2.5cm,
    headsep=0.2in
]{geometry}

\usepackage{fancyhdr}

\theoremstyle{plain}

\newtheorem{theorem}{Theorem}[section]
\newtheorem{subsec}{theorem}

\newtheorem{corollary}[theorem]{Corollary}

\newtheorem{proposition}[theorem]{Proposition}

\theoremstyle{definition}

\newtheorem{definition}[theorem]{Definition}

\newtheorem{example}[theorem]{Example}
\newtheorem{remark}[theorem]{Remark}
\newtheorem{cons}[theorem]{Construction}
\newcommand{\RR} {\mathbb R}
\newcommand{\CC} {\mathbb C}
\newcommand{\ZZ} {\mathbb Z}

\newcommand{\TT} {\mathbb T}

\newcommand{\QQ} {\mathbb Q}

\def\C{\mathrm{C}}

\def\uV{\underline{V}}

\def\uH{\underline{\mathrm{H}}}

\def\sO{\mathscr{O}}

\def\sB{\mathscr{B}}
\def\sN{\mathscr{N}}
\def\sM{\mathscr{M}}

\def\sA{\mathscr{A}}
\def\sE{\mathscr{E}}
\def\sI{\mathscr{I}}

\def\cM{\mathcal{M}}

\def\cI{\mathcal{I}}

\def\H{\mathrm{H}}
\def\S{\mathrm{S}}

\def\DGA{\mathrm{DGA}}
\def\GA{\mathrm{GA}}
\def\hom{\mathrm{Hom}}
\def\Aut{\mathrm{Aut}}

\DeclareMathOperator{\coker}{coker}
\DeclareMathOperator*{\colim}{colim}
\DeclareMathOperator{\Ext}{Ext}
\newcommand{\vc} {\mathrm{Vec}_G}

\def\c#1{\mathscr{#1}}
\def\s#1{\mathscr{#1}}
\def\u#1{\underline{#1}}

\def\ra{\rightarrow}
\def\xra{\xrightarrow}

\def\id{\mathrm{id}}

\titleformat{\subsection}[runin]
  {\normalfont\bfseries} % Text formatting (e.g., bold, large)
  {\thesubsection}             % Section number display
  {0.25em}                        % Space between number and title text
  {}                           % Code before title body
\titlespacing{\subsection}
  {0pt}                        % Left margin adjustment
  {3.25ex plus 1ex minus .2ex} % Vertical space before heading
  {1em}                        % Horizontal space between heading and following paragraph text

\begin{document}

	\title{Equivariant formality for rational $\C_{p^k}$-spaces}

	\author{ Rekha Santhanam\footnote{Department of Mathematics, Indian Institute of Technology Bombay, India. reksan@iitb.ac.in}, Bittu Singh\footnote{Department of Mathematics, Indian Institute of Technology Bombay, India. 22d0781@iitb.ac.in}, Soumyadip Thandar\footnote{School of Mathematics, TIFR, India.
 stsoumyadip@gmail.com} }

	\maketitle
    \begin{abstract}
         In this paper, we discuss equivariant formality for spaces equipped with a primary cyclic group action. We study the behavior of equivariant formality with some topological constructions. The algebraic complexity needs to deal with injectivity. In the end, we discuss equivariant intrinsic formality.
    \end{abstract}

	
\section{Introduction}\label{sec: intro}
In this article, we study the \emph{equivariant formality} of $ G$-rational spaces as defined by Fine and Triantafillou \cite{fine1993equivariant}. Elmendorf \cite{Elmendorf1983} showed that the  homotopy theory of $ G$-spaces can be studied in terms of presheaves in spaces on the orbit category $\sO_G$. Triantafillou \cite{TR1982} and Scull \cite{LS2002} use this perspective to show that the equivariant rational homotopy theory for finite groups and $\S^1$ action is equivalent to minimal system (injective diagrams on $\sO_G$) of differential graded algebras (DGAs). 

In general, the cohomology system of algebras of a system of $\DGA$s or a $G$-space cannot characterize its rational $G$-homotopy type. We say the diagram of DGAs or the corresponding space is equivariantly formal when the minimal model of the cohomology system of algebras of a space determines its equivariant rational homotopy type. 
In the non-equivariant setting, it is well known that several spaces, such as spheres, Eilenberg-MacLane spaces, 
and more generally, Hopf spaces, and compact K\"ahler manifolds are formal ( \cite{FHT,DGMS1975}). Formality is known to be compatible with several space-level constructions like products and wedges (\cite{FHT}). Also, the connected sum of formal simply connected compact manifolds is formal  (\cite[Theorem 3.13]{FOT}). 

In the equivariant case, Fine and Triantafillou \cite{fine1993equivariant} show that a compact K\"ahler manifold equipped with a holomorphic $G$-action is equivariantly formal. Rothenberg and Triantafillou, in  \cite{Rothenberg_Tri_BU} prove the equivariant formality of $BU(\alpha)$ for a complex representation $\alpha$, which is crucial to understand equivariant characteristic classes.  Triantafillou, in \cite{Tri_Hopf} establishes equivariant formality of  Hopf spaces with primary cyclic group actions. 
For spaces with $\S^1$ actions, Scull \cite{scull2004_T_formality} shows the equivariant formality of products and wedges of spaces with circle action induced from a circle action on the summands. Scull also proves that $\TT$-equivariant Eilenberg-MacLane spaces are equivariantly formal. 

In this article, we explore the equivariant formality of spaces for the finite group action case. 
Depending on the group action, a formal space may or may not be equivariantly formal.  Also, in the equivariant case, since the existence of a minimal model of a diagram of $\DGA$s requires an injectivity condition on the diagram, understanding equivariant formality is more challenging. It is therefore insufficient to define a $G$-space as equivariantly formal if its minimal model is simply the minimal model of its cohomology diagram. Fine and Triantafillou \cite{fine1993equivariant}  address this issue by first constructing the injective envelope of the cohomology system of algebras thought of as a diagram of DGAs with zero differential and then considering its minimal model.  

To compute or characterize minimal models in equivariant rational homotopy theory, we need to characterize the injective resolutions generated in each step of the minimal extension. To this end, we often restrict ourselves to the case when $G=\C_{p^k}$. Any injective resolution in this case has length one (\cref{cpk resolution length 1}) and hence is more amenable to computations. We use this fact in \cref{prop: explicit minimal cons cpk} to provide an explicit characterization of $\C_{p^k}$-minimal systems. This characterization makes the differences between non-equivariant minimality and equivariant minimality explicit. 

In the case of arbitrary finite groups, we utilize our knowledge of non-equivariant models to construct equivariant models under suitable conditions.  In \Cref{thm: minimal model for each level}, we show that a $\DGA$ endowed with a $G$-action admits a minimal model (in the sense of Sullivan) that naturally inherits a $G$-action. However, this may not always give rise to the correct algebraic model to capture the equivariant rational homotopy type of a $G$-space, as explained in \cref{rem: insufficiency to be a correct model}. We show that \cref{prop: lower fixed DGA minimal},  in certain special cases, this model is sufficient to determine the equivariant rational homotopy types. We also study the relation between levelwise minimality at a $G/H$ orbit and minimality of the system 
in  \cref{thm: levelwise minimal} and \cref{ex: minimal system but not levelwise minimal}.

With this setup in \cref{sec: equivariant formality} we study products, wedges, and connected sums on formal spaces which acquire a natural group action and their equivariant formality. 
Given a space $X$, there is a natural $\C _ {p^k}$- action on the $p^k$-fold product $X^{p^k}$ given by cyclic permutation of the factors. In \cref{permutation formal}, we show that if $X$ is formal, then this $\C_{p^k}$-space is equivariantly formal by describing its equivariant minimal model.
Similarly, for a space $X$ and for a finite group $G$, we can construct a $G$-space by taking the wedge of $|G|$ copies of $X$, where $G$ acts by permuting the summands. In \cref{thm: formality of G wedge}, we show that if $X$ is formal, then the $G$-space $\bigvee_{g\in G} X_g$ is equivariantly formal.

\cref{thm: retraction of eq formal} shows that an equivariant homotopy retract of an equivariantly formal space is equivariantly formal. This enables us to prove the converse direction of the results concerning the equivariant formality of various constructions.
In \cref{ex: join is not eqformal}, we consider the $\C_2$-space $X=\CC P^2 \# \CC P^2$, where the generator $\xi$ acts on the space by swapping the two copies. By computing its minimal model, we show that this space is not equivariantly formal. This is in contrast with the previous constructions and illustrates how the equivariant rational homotopy theory is fundamentally more involved.

A cohomology algebra $A^\ast$ is said to be \emph{intrinsically formal} if it is realized by a unique minimal Sullivan algebra up to isomorphism. Thus, intrinsic formality provides a stronger property than formality of a particular space: if $A^\ast=\H^\ast(X;\mathbb{Q})$ is intrinsically formal, then $X$ is formal. Intrinsic formality has been studied from several perspectives; see, for example, \cite{HalpStasheff1979,ShigaYagita,LUPT,ShigaYamguchi,SS12}.

The first and third authors introduced the notion of equivariant intrinsic formality in \cite{santhanam2025equivariant}. A cohomology system is said to be intrinsically formal if there is a unique minimal system of algebras up to isomorphism with the given cohomology system. This definition is closely related to equivariant formality. Any space whose cohomology system is intrinsically formal is formal. In \cref{sec: Intrinsic formality}, we study this concept for the group $\C_{p^k}$.  We describe a method for establishing sufficient conditions for plural homotopy types with a given cohomology system in \cref{thm: plural homotopy types in C}. We conclude this section with several examples and remarks highlighting how equivariant theory is significantly more complex than its classical counterpart.

\section{Background}\label{sec: background} 

Throughout this paper we work with $\DGA$s over $\QQ$ and assume $G$ is a finite group. We restrict our attention to finite-type $G$-simply connected CW-spaces. By a $G$-simply connected space, we mean a $G$-space $X$ such that $X^H$ are simply connected $(\pi_i(X^H)=0;$ for $ i=0,1)$ for each subgroup $H\leq G$. Recall a space is of \emph{finite type} if the cohomology of its fixed point sets over all subgroups $H$ is finite dimensional in each degree.
 
\subsection{Rational coefficient systems.}
We begin with coefficient systems and the notion of injectivity, which are essential for understanding equivariant rational homotopy theory. The following category serves as the indexing category for equivariant homotopy theory.

\begin{definition}\label{defn:orbit}
Given a group $G$, the category of canonical orbits is the category whose objects are $G$-sets $G/H$ and morphisms are $G$-maps between them. We denote this category by $\sO_G$. 
\end{definition}
 A \emph{coefficient system} for $G$ is a contravariant functor from $\sO_G$ to the category of abelian groups. If the codomain is the category of rational vector spaces, it is called a \emph{rational coefficient system}. We denote by $\vc$ the category of rational coefficient systems, whose morphisms are natural transformations.
 
A \emph{dual rational coefficient system} is a covariant functor from $\sO_G$ to the category of rational vector spaces. The corresponding category is denoted by $\vc^\ast.$

\begin{definition}\label{defn:inj object}    
An object $I$ in a category $\mathcal{C}$ is said to be \emph{injective} if, for every monomorphism $f:A\to B$ and every morphism $g:A\to I$, 
there exists a morphism $h:B\to I$ such that the following diagram commutes:
$$\xymatrix{
A \ar[d]_g\ar[r]^f& B\ar@{-->}[dl]^{h} \\
I }.$$
Equivalently, $h\circ f=g$. 
\end{definition}

The construction (see  \cite[Proposition 7.34]{LS2002}) below embeds a dual rational coefficient system into its \emph{injective envelope}
\begin{definition}\label{defn: standard inj env in vec}
Let $\underline{V}$ be a dual rational coefficient system.
For each subgroup $H\leq G$, define 
\begin{equation} 
V_H:=\bigcap_{H< K}\ker \u V(\hat{e}_{H,K}),
\end{equation}
 where $\hat{e}_{H,K}:G/H\to G/K$ is the canonical projection and $\u V(\hat{e}_{H,K})$ is the induced structure map on the functor $\u V$.  Note that  $V_G$ is defined to be $\uV(G/G)$. Let $\cI=\oplus_H \underline{V}_H$, where the direct sum is taken over all conjugacy classes of subgroups of $G$, and 
 \begin{equation}
 \underline{V}_H(G/K):=\hom_{\QQ(N(H)/H)}(\QQ(G/H)^K,V_H).
 \end{equation}
There is an monomorphism $\uV\to \cI$ extending the natural inclusions of $\bigcap_ {H< K}\ker \u V(\hat{e}_{H,K})\subseteq \cI (G/H)$.
\end{definition}

Triantafillou \cite[Section 3]{TR1982} proved that injective objects in the category of dual rational coefficient systems admit a particularly simple decomposition. In this case, the above embedding is an isomorphism.
\begin{proposition}\label{sum inj}
    A dual coefficient system $\u V$ is injective if and only if it is of the form $\u V = \bigoplus_H \underline{V}_H$ for some collection of $\QQ(N(H)/H)$-modules
$V_H$ and $$\underline{V}_H(G/K)=\hom_{\QQ(N(H)/H)}(\QQ(G/H)^K,V_H).$$
\end{proposition}

\subsection{System of DGAs.}
A \emph{diagram of $\mathrm{DGAs}$} (or an $\s O_G$-DGA) is a covariant functor from $\s O_G$ to the category of commutative differential graded algebras. Morphisms between two such diagrams are natural transformations. We note that various constructions on diagrams of $\DGA$s, such as the product $\oplus$ and coproduct $\otimes$, are defined levelwise; that is, they are obtained by taking the product and coproduct, respectively, at each object of the indexing category (see \cite[Chapter 3(c)]{FHT}).

Let $\u V$ be a dual rational coefficient system with a projective resolution $\u V^\ast$ and $\s A$ be a diagram of DGAs.
The double complex $\hom(\uV^\ast, \sA)$ gives a spectral sequence $$E_2^{s,t}=\Ext^s(\uV,\uH^t(\sA)) \Rightarrow \H^{s+t}(\sA;\uV).$$ Injectivity of $\sA $ is required for the convergence of the spectral sequence.  This spectral sequence plays a key role in the construction of minimal models, which will be discussed later in \cref{thm: existence of minimal model}.
\begin{definition}
    A \emph{system of $\mathrm{DGAs}$} $ \mathscr{A}$ is a covariant functor from the orbit category of $G$ to the category of $\mathrm {DGAs}$ such that the underlying dual rational coefficient system is injective.
\end{definition}

We now want to define weak equivalences on the category of diagrams of $\DGA$s. Before that, we define the notion of homotopy between morphisms of diagrams of $\DGA$s, which is dual to the notion of homotopy between equivariant maps of spaces.
\begin{definition}
    The morphisms $f,g: \sA \rightarrow \mathscr B$ are said to be \emph{homotopic} if there is a morphism $$F: \s A \rightarrow \mathscr B \otimes \u\QQ(t,dt)$$ where $|t|=0$, such that $p_0 F=f$ and $p_1 F= g$. Here $p_i:\s B \otimes \u{\QQ}(t,dt) \rightarrow \mathscr B$ are defined by $p_i(t)=i$ and $p_i(dt)=0$ for $i=0,1.$
\end{definition}
This does not define an equivalence relation on the diagram of $\DGA$s. However, it does give an equivalence relation on {\it minimal systems of $\DGA$s} (Definition \ref{def:minimal}), which we discuss later in this section. We define a coarser relation on diagrams of $\DGA$s. Given two diagrams of $\DGA$s $\sA, \sB$,  if there is a morphism  $f$ in $\sA \to \sB$ or $\sB \to \sA$ inducing a cohomology isomorphism at each level (at $G/H$ for all $H \leq G$), then $f$ is said to be a quasi-isomorphism. The equivalence relation generated by quasi-isomorphisms is defined as a weak equivalence of systems of $\DGA$s.

\begin{definition}\label{def:weakyeqivalent}
Let $\sA$ and $\sB$ be two $\DGA$ diagrams over $\sO_G$. We say $\sA$ and $\sB$ are  \emph{weakly equivalent} if there is a zig-zag of quasi-isomorphisms connecting $\s A$ and $\s B.$
\end{definition}

\subsection{Minimal models.}
The algebraic counterparts of equivariant rational Postnikov towers are minimal systems of $\DGA$s. Consequently, minimal systems encode the rational homotopy types of $G$-spaces. In the equivariant setting, however, the additional injectivity requirement makes the notion of minimality considerably more subtle.

A \emph{(Sullivan) minimal algebra} $(\QQ (V), d)$ is a freely generated commutative differential graded algebra on a graded vector space $V$ such that the differential decomposes as $d(V)\subseteq\QQ^+(V)\cdot\QQ^+(V)$ (see \cite{Sullivan_infinitesimal}). 
 Recall that a \emph{minimal algebra} $M$ is a free graded algebra which can be written as an increasing union of $M_n$'s where $M_0=M_1=\QQ$, and $M_{n-1} \hookrightarrow M_n$ is a Hirsch extension for every $n$ (see \cite[Theorem 10.3]{PGJM}). A {\it minimal model} of a $1$-connected $\DGA$ $A$, is a pair $(M,\rho)$, where $M$ is a minimal algebra and $\rho:M\to A$ is a quasi-isomorphism. 
\subsubsection{Elementary extension of a system of DGAs}\label{construction of elementary extn}
Given a system of $\DGA$s $\sA$, a system of vector spaces $\underline{V}$ and a morphism $\alpha : \underline{V} \rightarrow \underline{Z}^{n+1}(\sA)$ in  $\vc^\ast$, we can extend the system of $\DGA$s to form a new system of  $\DGA$s $\sA^\alpha(\underline{V})$. This extension $\sA \hookrightarrow \sA^{\alpha}(\uV)$ is called an \emph{elementary extension} and the construction is discussed below.
\underline{{\textbf{Construction:}}}
\begin{enumerate}
     \item We will take the injective resolution of $\underline{V}:$ $$\begin{tikzcd}
	\underline{V} & \underline{V}_0 & \underline{V}_1 & \underline{V}_2 & \cdots
	\arrow[hook, from=1-1, to=1-2]
	\arrow["{w_0}", from=1-2, to=1-3]
	\arrow["{w_1}", from=1-3, to=1-4]
	\arrow["{w_2}", from=1-4, to=1-5]
\end{tikzcd}$$
    We embed $\underline{V}$ into its injective envelope of $\underline{V}$ (say $\underline{V_0}$), then embed $\underline{V_0}/\underline{V}$ into its injective envelope (say $\underline{V_1}$). Define $w_0:\underline{V_0} \rightarrow \underline{V_1}$ by the composition. Now we will repeat the process, i.e., take the injective envelope of $\coker  w_i$.
    $$\begin{tikzcd}
	{\underline  V} & {\underline  V_0} & {\underline  V_1} & {\underline  V_2} & {\underline V_3\cdots} \\
	& {\underline   V_0 /\underline   V} & {\coker  w_0} & {\coker  w_1}
	\arrow[hook, from=1-1, to=1-2]
	\arrow["{w_0}", dashed, from=1-2, to=1-3]
	\arrow[from=1-2, to=2-2]
	\arrow["{w_1}", dashed, from=1-3, to=1-4]
	\arrow[from=1-3, to=2-3]
	\arrow["{w_2}", dashed, from=1-4, to=1-5]
	\arrow[from=1-4, to=2-4]
	\arrow[hook, from=2-2, to=1-3]
	\arrow[hook, from=2-3, to=1-4]
	\arrow[hook, from=2-4, to=1-5]
\end{tikzcd}$$
\item We construct a commutative ladder
$$\begin{tikzcd}\label{ladder diagram}
	\underline{V} & \underline{V}_0 & \underline{V}_1 & \underline{V}_2 & \cdots \\
	{{Z}^{n+1}(\mathscr A)} & {\mathscr A^{n+1}} & {\mathscr A^{n+2}} & {\mathscr A^{n+3}} & \cdots
	\arrow[hook, from=1-1, to=1-2]
	\arrow["\alpha"', from=1-1, to=2-1]
	\arrow["{w_0}", from=1-2, to=1-3]
	\arrow["{\alpha_0}"', from=1-2, to=2-2]
	\arrow["{w_1}", from=1-3, to=1-4]
	\arrow["{\alpha_1}"', from=1-3, to=2-3]
	\arrow["{w_2}", from=1-4, to=1-5]
	\arrow["{\alpha_2}"', from=1-4, to=2-4]
	\arrow[hook, from=2-1, to=2-2]
	\arrow["d", from=2-2, to=2-3]
	\arrow["d", from=2-3, to=2-4]
	\arrow["d", from=2-4, to=2-5]
\end{tikzcd}$$
We produce the maps $\alpha_i$ inductively by observing that $d \alpha_i w_{i-1}=d d \alpha_{i-1}=0$, so $d \alpha_i|_{\text {im } w_{i-1}}=0$ and by the injectivity of $\mathscr{A}$ we can fill in
$$\begin{tikzcd}
	{\underline{V}_i/\text{im } w_{i-1}} & {\sA^{n+i+1}} \\
	{\underline{V}_{i+1}}
	\arrow[from=1-1, to=1-2]
	\arrow["{d\alpha_i}", hook', from=1-1, to=2-1]
	\arrow["{\alpha_{i+1}}"', dashed, from=2-1, to=1-2]
\end{tikzcd}$$
\item Let ${\QQ}(\uV_i)$ denotes the freely generated commutaive $\DGA$ ${\QQ}(\uV_i(G/H))$ at each orbit levels $G/H.$
Define $\mathscr A^{\alpha}(\underline{V}):=\sA \otimes \bigotimes_i (\QQ(\underline{V}_i))$ and extend the differential by defining $d|_{\underline{V}_i}=w_i+(-1)^i \alpha_i $. Here $\underline{V}_i$ are added in $n+i$ degree. The injectivity of the $\mathscr A^{\alpha}(\underline{V})$ system is clear, as it is a tensor product of injective systems (see \cite{Golasinski1997}). We write $\sA(\uV)$ instead of $\sA^{\alpha}(\uV)$ for simplicity.
\end{enumerate}

  A minimal system of $\DGA$s is defined as follows. 
\begin{definition}\label{def:minimal}
A system of $\DGA$s $\sM$ is \emph{minimal} if $\sM=\cup_n\sM_{n}$, where $\sM_{0}=\sM_{1}=\underline{\mathbb{Q}}$ and $\sM_{n}=\sM_{n-1}(\underline{V})$ is the elementary extension for some diagram of vector spaces $\underline{V}$ of degree $\geq n$. 
\end{definition}

The following lifting criterion is often useful and needed to show that homotopy between minimal systems reduces to isomorphisms. 
 \begin{proposition}\cite[Proposition 5.4]{TR1982} \label{prop: lifting property}
     If $g$ is a quasi-isomorphism, then the following diagram admits a lift that is unique up to homotopy:
     $$\begin{tikzcd}
	& {\mathscr A} \\
	{\mathscr M} & {\mathscr B}
	\arrow["g", from=1-2, to=2-2]
	\arrow["{\tilde f}", dashed, from=2-1, to=1-2]
	\arrow["f", from=2-1, to=2-2]
\end{tikzcd}$$
 \end{proposition}

If we have two minimal systems $\sM$, $\sN$ and quasi-isomorphisms $\rho_1:\sM\to \sA$ and $\rho_2:\sN\to \sA$ by the lifting property of maps from minimal systems to systems of $\DGA$s, we get a map $f:\sM\to \sN$ which is a quasi-isomorphism. By Theorem \ref{thm: uniqueness of minimal alg}, we get $f\simeq g$ where $g$ is an isomorphism.

\begin{theorem}\label{thm: uniqueness of minimal alg}\cite[Theorem 3.8]{LS2002} 
If $f:\sM\to \sN$ is a quasi-isomorphism between minimal systems of $\DGA$s, then $f$ is homotopic to an isomorphism $g:\sM \ra \sN$.

\end{theorem}

This leads to the definition of the minimal model unique up to isomorphism.

\begin{definition}\label{def:minimalmdel}
If $\sM$ is a minimal system and $\rho: \sM\to \sA$ is a quasi-isomorphism, we say that $\sM$ is a \emph{minimal model} of $\sA$. 
\end{definition}
Observe that any minimal system is cohomologically $1$-connected, that is, it satisfies $\uH^0(\sM)=\mathbb{Q}$ and $\uH^1(\sM)=0$. Triantafillou \cite[Theorem 1.4]{TR1982} and Scull \cite[Theorem 3.11]{LS2002} establish the existence of minimal model.
\begin{theorem}\label{thm: existence of minimal model}
    A system of $\DGA s$, $\mathscr A$, that is cohomologically $1$-connected has a minimal model, that is, there exists $(\mathscr M,\rho)$, where $\mathscr M$ is a minimal system and $\rho:\mathscr M\to \sA$ is a quasi-isomorphism.
\end{theorem}
Scull proved this theorem by inductively constructing the minimal model. Starting with the constant coeffient system $\u \QQ$, the $n$-th stage $\mathscr M_n$ is obtained from $\mathscr M_{n-1}$ by an elementary extension determined by a suitable morphism
$\underline{V}\rightarrow Z^{n+1}(\mathscr M_{n-1}),$
where
$\underline{V}=\underline{\H}^{n}(\mathscr M_{n-1},\mathscr A).$
The minimal model is then obtained as the colimit of this inductive construction. Since minimal models are unique up to isomorphism, we shall henceforth refer to \emph{the minimal model} of a system.

Given a $G$-space $X$, the \emph{PL de Rham system} $\u {\s E}(X)$ is defined by assigning to each orbit $G/H$ the algebra of piecewise linear polynomial differential forms on the corresponding fixed point spaces: $\u {\s E}(X)(G/H)=\s E_{X^H}$. Triantafillou \cite[Proposition 4.3]{TR1982} has shown that $\u {\s E}(X)$ is injective in $\vc^\ast.$
Consequently, the system $\u {\s E}(X)$ admits a minimal model $\s M_X$. 
\begin{definition}
Let $X$ be a $G$-space, then the minimal model of $X$ is the minimal model of the system $\u {\s E}(X)$
\end{definition}

\subsection{Injective envelopes.}
A diagram of $\DGA$s may not admit a minimal model unless an additional algebraic condition, namely injectivity in the category of dual rational coefficient systems, is imposed. To overcome this obstruction, Fine and Triantafillou \cite{fine1993equivariant} introduced the notion of an \emph{injective envelope} of a diagram of $\DGA$s. This construction replaces a given diagram by a quasi-isomorphic one that satisfies the required injectivity condition and therefore admits a minimal model.
We now describe their construction.
  \begin{definition}\label{def:associ}
Let $G$ be a group and $H\leq G$. Let $A_H$ be a $\DGA$ over $\QQ$ such that $N(H)/H$ acts on it by $\DGA$ automorphisms. The \emph{associated system of $\DGA$s} $\sA_H$, of the $\DGA$ $A_H$, is a system of $\DGA$s defined as follows:
Let $V_H$ be a copy of $A_H$ considered as a graded $\QQ(N(H)/H)$-module by forgetting the differential, and let $\underbar{V}_H$ be the induced injective diagram of vector spaces \Cref{defn: standard inj env in vec}. Let $s\underbar{V}_H$ be copy of $\underbar{V}_H$ with a shift of degree by $+1$. We denote by $\bigwedge_H$ the system of acyclic $\DGA$s generated by $\underbar{V}_H\oplus s\underbar{V}_H$, where $d(\underbar{V}_H)=s\underbar{V}_H$. Now we define the {associated system} $\sA_H$ by
$$\sA_H(G/K)=
\begin{cases}
     \bigwedge_H(G/K), & \text{for } (K)=(H)\\
    \hom_{\QQ(N(H)/H)}(\QQ(G/H)^K, A_H),  & \text{ for } (K) \neq (H)
\end{cases}$$
where $(H)$ is the conjugacy class of $H$ in $G$. The value of this functor on a morphism is obvious. 
\end{definition}

\begin{definition}\label{def: inj enlargement}
Let $\sA$ be a system of $\DGA$s and let $A_H$ be the subalgebra of $\sA(G/H)$ which is equal to $\bigcap _{H< H'}\ker \sA(\hat e_{H,H'})$, where $\sA(\hat e_{H,H'})$ is the morphism induced by the projection $G/H\to G/H'$. Let $\sA_H$ be the associated system to $A_H$. The enlargement of $\sA$ at $H$ is the system of $\DGA$s $\sI_H(\sA)$ defined by 
$$\sI_H(\sA)(G/K)=
\begin{cases}
     \sA(G/K)\otimes \sA_H(G/K), & \text{for } K<H\\
    \sA(G/K),  & \text{otherwise } 
\end{cases}$$
where $K<H$ means that $K$ is a proper subgroup of a conjugate of $H$. The value of the functor $\sI_H(\sA)$ on morphisms is the obvious one, namely, they are equal to the old morphisms when restricted to the subsystems $\sA$ and $\sA_H$ respectively.

The injective envelope of a system of $\DGA$s $\sA$ is constructed by successively taking injective enlargement at each subgroup up to conjugacy (one subgroup for each conjugacy class).

\end{definition}

\begin{theorem}\cite[Theorem 1]{fine1993equivariant}\label{thm:fine inj extension}
For a diagram of $\DGA$s $\sA$ over $\sO_G$, where $G$ is finite group, there is an injective system  of $\DGA$s  $\sI(\sA)$, called the injective envelope of $\sA$ along with an inclusion $i:\sA\to \sI(\sA)$ which is a quasi-isomorphism. 
\end{theorem} 

\begin{proposition}\cite[Proposition 8]{fine1993equivariant}\label{prop: fine factor through inj}
    Let $f:\s A\ra \s B$ be a map of $\s O_G$-$\DGA$s and $\s B$ is injective in $\vc^\ast$. Then the map factors through the injective envelope $i_H:\s A\ra \s I_H(\s A)$ at $H$, that is, there exists $\tilde f:\s I_H(\s A)\ra \s B$ such that $\tilde f i_H=f,$ that is we have the following commutative diagram:
    $$\begin{tikzcd}
	\sA & \sB \\
	{\sI_H(\sA)}
	\arrow["f", from=1-1, to=1-2]
	\arrow["{i_H}"', from=1-1, to=2-1]
	\arrow["{\tilde f}"', dashed, from=2-1, to=1-2]
\end{tikzcd}$$
\end{proposition}
Thus, for a morphism $g:\s A\to \s B$ between diagrams of DGAs, we have an extension $\sI (g):\s I(\s A)\to \s I(\s B)$ between systems of DGAs.

\section{Minimal models for \texorpdfstring{$\sO_{\C_{p^k}}$}{O primary cyclic group}-DGAs}\label{sec: Minimal models for Cpk}
In this section, we study the properties of minimal models of $\sO_{\C_{p^k}}$-$\DGA$s. We first discuss the characterization of an injective coefficient system for $\sO_{\C_{p^k}}$ in terms of its structure maps.

Given a finite group  $G$  and 
$\uV$ injective object in $\vc^\ast$, the structure maps $\uV(\hat e_{K,K'})$ for the subgroups $K' \leq K$ are surjective. 
 For example, to verify this claim for abelian groups, it is enough to show that this holds for $\underline {V}_H$ by \cref{sum inj}. 
    For an abelian group $G$, the fixed point set of the coset $G/H$ is given by

    \begin{equation} \label{eq: coset stabilizer}
        (G/H)^K= \begin{cases}
        G/H, \text{ if } K \leq H\\ \emptyset, \text{ otherwise} 
        \end{cases}
    \end{equation}
 We have \begin{align}
     \uV_H (G/K)&= \hom_{\QQ(NH/H)}(\QQ(G/H)^K, V_H) \nonumber\\
 &=\hom_{\QQ(G/H)}(\QQ(G/H)^K, V_H)\nonumber \\
 &= \begin{cases}
     V_H, &\text{if } K \leq H \\
     0, &\text{ otherwise.}  
 \end{cases}~~~~~~~~~~~~~~~~~~~~~~~\text{          by \eqref{eq: coset stabilizer}} \label{eq: injective on cosets}  \end{align}
Applying the covariant functor $\uV_H$ to the projections $\hat e_{K,K'} : G/K \rightarrow G/K'$ for $ K \leq K'$, we obtain the structure map $$\uV_H(\hat e_{K,K'}) : \uV_H(G/K) \longrightarrow \uV_H(G/K')$$
Using \eqref{eq: injective on cosets}, these maps take one of the following forms:
 \begin{align*}
     &\uV_H(\hat e_{K,K'}) : V_H \longrightarrow V_H, &\text{if } K' \leq H\\
     &\uV_H(\hat e_{K,K'}) : \uV_H(G/K) \longrightarrow 0, &\text{otherwise.}
 \end{align*}
 The first induced map is the identity and the second is zero. Hence the structure maps are surjective.

In particular, when $G=\C_{p^k}$, the condition that the structure maps are surjective is sufficient.   A dual proof for this characterization of injectivity is given by Triantafillou \cite{Tri_Hopf}. For completeness and clarity, we provide a direct proof in the following proposition. 

\begin{proposition}\label{thm: injectivity in Z_p^n}
Let $G=\C_{p^k}$ where $p$ is prime. Then, a diagram of vector spaces
$\u U$ is injective as an element of $\vc^\ast$  if and only if the structure maps $\u U(\hat{e}_{p^i,p^{i-1}})$ are surjective for $1\leq i \leq k.$  Dually, a diagram of vector spaces
$\u U$ is projective as an element of $\vc$  if and only if the structure maps $\u U(\hat{e}_{p^i,p^{i-1}})$ are surjective for $1\leq i \leq k.$    
\end{proposition}
\begin{proof}
We have the following chain of canonical projections:
    \begin{equation}\label{eq: lattice Z_p^n}
            G/p^k \xra{\hat{e}_{p^k,p^{ k-1}}}G/p^{k-1}
\xra{}\cdots \xra{} G/p^{1} \xra{\hat{e}_{p^1,p^{ 0}}}  G/p^{0}.
    \end{equation} 
  
     Let the structure maps $\u U(\hat{e}_{p^i,p^{i-1}})$ be surjective for $1\leq i \leq k.$     
We will show that its injective envelope $\uV$ of $\u U$ is isomorphic to $\u U$ itself, which will conclude the proof. By the construction of the injective envelope, we have $\uV= \bigoplus_{(H)} \uV_H$, where $V_H$ are given by 
\begin{align*}
    &V_H= \bigcap_{K> H} \ker \u U (\hat e_{H,K}), &\text{when } H \neq G\\
    &V_G=\u U(G/G).
\end{align*}
Let  $N_{r,s} = \ker \u U (\hat e_{p^r,p^s})$ for $r>s$. 
 Functoriality of $\u U$ gives $$\u U(\hat e_{p^r,p^s}) =\u U(\hat e_{p^{s+1},p^s}) \circ \u U(\hat e_{p^{s+2},p^{s+1}}) \circ \dots \circ\u U(\hat e_{p^{r-1},p^{r-2}}) \circ \u U(\hat e_{p^r,p^{r-1}}).$$ Therefore $N_{r,s}\subseteq  N_{r,s'}$ for $s\geq s'$ 
 which allows us to obtain
\begin{align}
    &V_{p^r}=\bigcap_{s<r } N_{r,s}=N_{r,r-1}, &\text{where } r=1, \dots , k \label{eq: intersection absorbed 1} \\
    &V_{p^0}= \u U (G/G).\nonumber
\end{align}
It follows that
\begin{align*}
    {\uV_{p^r}}(G/p^i)&=\hom_{\QQ(G/{p^{r}})}(\QQ(G/{p^{r}})^{(p^i)}, {V_{p^r}})\\
&=\begin{cases}
    V_{p^r} , &\text{ if } i \geq r\\
    0 , &\text{ if } i < r.
\end{cases}
\end{align*}
Hence, \begin{align}
   {\uV}(G/p^i)&=\bigoplus_{r=0}^k {\uV_{p^r}}(G/p^i)=\bigoplus_{r\leq i} {{V_{p^r}}} \nonumber \\
   &=\big(\bigoplus_{r=1}^i {N_{r,r-1}}\big) \oplus \u U(G/G). \label{eq: inj sumsplits}
\end{align}
    Applying $\u U$ on \eqref{eq: lattice Z_p^n}, we have
\begin{equation}\label{eq: A-lattice Z_p^n}
           \u U(G/p^k) \xrightarrow{\u U(\hat{e}_{p^k,p^{ k-1}})} \u U(G/p^{k-1})
\longrightarrow \cdots \longrightarrow \u U(G/p^1 ) \xrightarrow{\u U(\hat{e}_{p^1,p^{0}})} \u U(G/p^0) . \end{equation}
Since each arrow is surjective, we have
\begin{equation}\label{eq: A as sumsplits}
   \u U(G/p^i) \cong \big(\bigoplus_{r=1}^i {N_{r,r-1}}\big) \oplus \u U(G/G).
\end{equation}
Comparing \eqref{eq: inj sumsplits} and \eqref{eq: A as sumsplits}, we have $\u U \cong \uV.$
\end{proof}

\begin{corollary}
     Let $X$ be a $G=\C_{p^k}$-space. Consider the  $X^m$ as a $G$-space with the diagonal $G$-action. If $\uH^\ast(X)$ is injective then so is $\uH^\ast(X^m)$. 
\end{corollary}
\begin{proof}    By hypothess, the cohomology diagram $\s A=$$\uH^\ast(X)$, that is, $\s A(G/{p^i})=\H^\ast(X^{(p^i)})$, is injective, therefore, the structure maps 
       $\s A(\hat e_{p^i, p^{i-1}}): \s A (G/p^i) \rightarrow \s A(G/p^{i-1}) 
    $ are surjective, for $i=1 , \dots, k$ (using \Cref{thm: injectivity in Z_p^n}).
   Let $\mathscr B=$$\uH^\ast(X^m)$. Applying the K\"unneth formula, it follows that the map induced by $\s B(\hat e_{p^i, p^{i-1}}):\mathscr B (G/p^i) \rightarrow \mathscr B(G/p^{i-1}) $ can be identified with $ \s A(\hat e_{p^i, p^{i-1}})^{\otimes m}: \s A (G/p^i) ^{\otimes m}\rightarrow \s A(G/p^{i-1}) ^{\otimes m}$ which is surjective. Hence, by \cref{thm: injectivity in Z_p^n}, $\uH^\ast(X^m)$ is also injective in $\vc^\ast.$ 

\end{proof}
In the following example, we explicitly prove non-injectivity of the given diagram of $\DGA$s, thus noting the ease of the above characterization of injectivity. 

\begin{example}
Consider the rational coefficient system $\u V$ in $\vc^\ast$ for $G=\C_{p^2}$ given by
$$\u V=
\begin{cases}
\QQ(x), & \text{at } \C_{p^2}/e,\\
\QQ(x,x'), & \text{at } \C_{p^2}/\C_p,\\
\QQ(x'), & \text{at } \C_{p^2}/\C_{p^2},
\end{cases}$$
Suppose the structure maps are given by $
\u V(\hat e_{e,\C_p}) : x\mapsto x,
$
and
$
\u V(\hat e_{\C_p,\C_{p^2}}): x\mapsto 0,$ $
x'\mapsto x'.
$
Since $\u V(\hat e_{e,\C_p})$ is not surjective, $\u V$ is not injective in
$\vc^\ast$. To verify this directly, consider the monomorphism
$j:\u A\hookrightarrow \u B$ given by the inclusion
$$\u A=
\begin{cases}
(0), & \text{at } \C_{p^2}/e,\\
\QQ(a), & \text{at } \C_{p^2}/\C_p,\\
\QQ(a), & \text{at } \C_{p^2}/\C_{p^2},
\end{cases}
\qquad\longrightarrow\qquad
\u B=
\begin{cases}
\QQ(a), & \text{at } \C_{p^2}/e,\\
\QQ(a), & \text{at } \C_{p^2}/\C_p,\\
\QQ(a), & \text{at } \C_{p^2}/\C_{p^2}.
\end{cases}$$
where the structure maps are the obvious injections.
Define a morphism $f:\u A\to\u V$ by sending the generator $a$ to $x'$ at each
orbit. We claim that $f$ admits no extension
$\tilde f:\u B\to\u V$, that is, there is no lift making the diagram
$$\begin{tikzcd}
	{\u A} & {\u V} \\
	{\u B}
	\arrow["f", from=1-1, to=1-2]
	\arrow["j"', hook, from=1-1, to=2-1]
	\arrow["{\tilde f}"', dashed, from=2-1, to=1-2]
\end{tikzcd}$$
commutes. Indeed, if such a lift $\tilde f$ existed, then the naturality of $\tilde f$
with respect to the morphism
$\hat e_{e,\C_p}:\C_{p^2}/e\to\C_{p^2}/\C_p$
would force the following diagram to commute:
$$\begin{tikzcd}
	0 && 0 \\
	a & a & x' \\
	& a
    \arrow[draw=green, from=1-1, to=1-3, bend left=20]
	\arrow[from=1-1, to=2-1, "{\u A(\hat e_{e,\C_p})}"']
	\arrow[draw=blue, from=1-1, to=2-2, bend left=15]
	\arrow[from=1-3, to=2-3, "{\u V(\hat e_{e,\C_p})}"]
	\arrow[draw=green, from=2-1, to=2-3, bend left=20]
	\arrow[draw=blue, from=2-1, to=3-2, bend left=15]
	\arrow[draw=red, dashed, from=2-2, to=1-3, bend left=15]
	\arrow[from=2-2, to=3-2, "{\u B(\hat e_{e,\C_p})}"]
	\arrow[draw=red, dashed, from=3-2, to=2-3, bend left=15]
\end{tikzcd}$$
However, this is impossible, since the square involving $\tilde f(\C_{p^2}/e)$ and $\tilde f(\C_{p^2}/\C_p)$ (red arrows) can not commute. This contradiction
shows that no such lift exists. Hence $\u V$ is not injective in $\vc^\ast$.
\end{example}
We now show that injective resolutions of $\sO_{\C_{p^k}}-\DGA$s simplify considerably. This simplification is a key ingredient in obtaining an explicit description of $\C_{p^k}$-minimal systems. 
\begin{proposition}\label{cpk resolution length 1} 
    Let $\uV$ be a $\C_{p^k}$-dual rational coefficient system. Then $\u V$ admits an injective resolution of length $1.$
\end{proposition}

\begin{proof}
    
     Let $w:\uV\hookrightarrow \uV_0$ be the injective envelope of $\uV.$ By injectivity of $\uV_0$ in $\vc^\ast$, the map $\uV_0(\hat e_{H,K}): \uV_0(G/H) \ra \uV_0(G/K)$ is surjective for all subgroups $H\leq K$ of $G=\C_{p^k}$. By commutativity of the following diagram:
     $$\begin{tikzcd}
 	{\uV_0(G/H)} & {\uV_0(G/H)/w(G/H)(\uV(G/H))} \\
 	{\uV_0(G/H)} & {\uV_0(G/K)/w(G/K)(\uV(G/K))}
 	\arrow[from=1-1, to=1-2]
 	\arrow["{\uV_0(\hat e_{H,K})}", from=1-1, to=2-1]
 	\arrow["{\uV_0(\hat e_{H,K})}", from=1-2, to=2-2]
 	\arrow[from=2-1, to=2-2],
 \end{tikzcd}$$
     we have $\coker w(G/H) \ra \coker w(G/K)$ is surjective. Hence, $\coker w$ is injective in $\vc^\ast.$ Therefore, we have the following injective resolution $\uV\hookrightarrow \uV_0 \ra \u V_1\ra 0$, where $\u V_1= \coker w$. This completes the proof.
 \end{proof}
\begin{remark}\label{explicit form of cpk minimal system}
    Note that the above proposition gives a more compact description of the elementary extension when $G=\C_{p^k}$.  For the $n$-th stage of the minimal system construction, we take an elementary extension of $\s M_{n-1}$ using some $\uV_n$ in $\vc^\ast$. Let $\uV_{n,0}$ and $\uV_{n,1}$ be obtained from the injective resolutions.
Then the commutative ladder in the elementary extension construction \cref{construction of elementary extn}   reduces to
   $$ \begin{tikzcd}
	{\underline{V}_n} & {\underline{V}_{n,0}} & {\underline{V}_{n,1}} & 0 \\
	{{Z}^{n+1}(\mathscr M_{n-1})} & {\mathscr M_{n-1}^{n+1}} & {\mathscr M_{n-1}^{n+2}} & {\mathscr M_{n-1}^{n+3}}
	\arrow[hook, from=1-1, to=1-2]
	\arrow["\alpha"', from=1-1, to=2-1]
	\arrow["{w_0}", from=1-2, to=1-3]
	\arrow["{\alpha_0}"', from=1-2, to=2-2]
	\arrow[from=1-3, to=1-4]
	\arrow["{\alpha_1}"', from=1-3, to=2-3]
	\arrow[hook, from=2-1, to=2-2]
	\arrow["d", from=2-2, to=2-3]
	\arrow["d", from=2-3, to=2-4]
\end{tikzcd}$$
Therefore, for a system of $\C_{p^k}$-$\DGA$s, the minimal models are  of the form:
    $ \sM = \bigcup_n \sM_n, $
    where $\sM_n = \QQ(\uV_{2,0} \oplus \uV_{2,1} \oplus \uV_{3,0} \oplus \uV_{3,1} \oplus \dots \oplus \uV_{n,0} \oplus \uV_{n,1})$. Here, the degrees are given by $|\uV_{n,j}| = n+j$, and the differentials satisfy $d(\uV_{n,0}) \subseteq \sM_{n-1}^{n+1} \oplus \uV_{n,1}$ and $d(\uV_{n,1}) \subseteq \sM_{n-1}^{n+2}$. More explicitly, $d|_{\uV_{n,0}} = \alpha_0 + w_0$ and $d|_{\uV_{n,1}} = -\alpha_1$. Furthermore, since $\uV_{n,1} \cong \uV_{n,0}/\ker w_0$, The map  $\alpha_1$ can be identified with the map induced by passing $d\circ\alpha_0:\u V_{n,0}\to \s A^{n+2}$ to the quotient.
\end{remark}
Therefore, we obtain the following explicit characterization of a $\C_{p^k}$-minimal system.

\begin{proposition}\label{prop: explicit minimal cons cpk}
    A $\C_{p^k}$-minimal system $\s M$ is a system of $\DGA$s freely generated by injective rational coefficient systems of the form:
    $ \s M = \QQ\left( \bigoplus_{i\geq 0} \bigoplus_{j=0,1} \u V_{i,j} \right) .$ That is, $\s M = \bigcup_n \s M_n$, where $\s M_0 = \s M_1 = \u\QQ$ and $\s M_n = \s M_{n-1} \otimes \QQ(\u V_{n,0} \oplus \u V_{n,1})$ with degrees $|\u V_{i,j}| = i+j$, such that the differential $d$ satisfies:
    \begin{enumerate}
        \item $d(\u V_{n,0}) \subseteq \s M_{n-1}^{n+1} \oplus \u V_{n,1}$ such that $p_2 \circ d: \u V_{n,0} \to \u V_{n,1}$ is surjective and $\u V_{n,0}$ is an injective envelope of $\ker(p_2 \circ d)$.
        \item $d(\u V_{n,1})$ is induced by the composition:
        $ \u V_{n,0} \xrightarrow{p_1 \circ d} \s M_{n-1}^{n+1} \xrightarrow{-d} \s M_{n-1}^{n+2} $
    \end{enumerate}
    where $p_i$ denotes the projection onto the $i$-th component of $\s M_{n-1}^{n+1} \oplus \u V_{n,1}$ for $i = 1, 2$.
\end{proposition}

\begin{proof}
    It follows from \cref{explicit form of cpk minimal system} that a minimal system satisfies the mentioned conditions.
    
    Conversely, assume $\s M$ satisfies these conditions. In this case, we claim that $\s M_n$ is obtained by an elementary extension $\s M_{n-1}^\alpha(\u V_n)$, where $\u V_n = \ker(p_2 \circ d: \u V_{n,0} \to \u V_{n,1})$ and $\alpha$ is given by the composition $\u V_n \hookrightarrow \u V_{n,0} \xrightarrow{p_1 \circ d} \s M_{n-1}^{n+1}$.

    Let $w_0 = p_2 \circ d$, $\alpha_0 = p_1 \circ d$, and $\alpha_1 = -d|_{\u V_{n,1}}$. From the information given in the conditions, we obtain the following commutative diagram:
    $$\begin{tikzcd}
        {\ker w_0} & {\u V_{n,0}} & {\u V_{n,1}} \\
        {Z^{n+1}(\s M_{n-1})} & {\s M_{n-1}^{n+1}} & {\s M_{n-1}^{n+2}}
        \arrow[hook, from=1-1, to=1-2]
        \arrow[dashed, "\alpha"', from=1-1, to=2-1]
        \arrow["w_0", two heads, from=1-2, to=1-3]
        \arrow["\alpha_0"', from=1-2, to=2-2]
        \arrow["\alpha_1"', from=1-3, to=2-3]
        \arrow[from=2-1, to=2-2]
        \arrow["d"', from=2-2, to=2-3]
    \end{tikzcd}$$
    We note that the restriction of $\alpha_0$ to $\ker w_0$ lands in $Z^{n+1}(\s M_{n-1})$, since
     $d \circ \alpha_0(\ker w_0) = \alpha_1 \circ w_0(\ker w_0) = 0.$
    Therefore, from the construction in \cref{construction of elementary extn}, it follows that $\s M_n$ is obtained via an elementary extension through $\alpha: \u V_n \to Z^{n+1}(\s M_{n-1})$.
\end{proof}

\begin{example}
We illustrate the explicit construction of a $\C_{p^2}$-minimal system given an injective coefficient system with differentials satisfying the conditions of \cref{prop: explicit minimal cons cpk}. We start with $\s M_0=\s M_1=\u\QQ.$

Let
$$\u V_{2,0}=
\begin{cases}
(x), & \text{at } \C_{p^2}/e,\\
(x), & \text{at } \C_{p^2}/\C_p,\\
(x), & \text{at } \C_{p^2}/\C_{p^2},
\end{cases}
\qquad
\u V_{2,1}=
\begin{cases}
(y), & \text{at } \C_{p^2}/e,\\
(y), & \text{at } \C_{p^2}/\C_p,\\
0, & \text{at } \C_{p^2}/\C_{p^2},
\end{cases}$$
where the structure maps are the obvious surjections, identity and the zero maps appropriately. Assume that the differential on $\u V_{2,0}$ is given by
$dx=y$
at $\C_{p^2}/e$. Naturality conditions will imply 
$$dx=y \quad\text{at } \C_{p^2}/\C_p,
\qquad
dx=0 \quad\text{at } \C_{p^2}/\C_{p^2}.$$
The differential on $\u V_{2,1}$ is then forced to be zero by the commutativity of the following diagram:
$$\begin{tikzcd}
	{\u V_{2,0}} & {\u V_{2,1}} \\
	{\s M_1^3=0} & {\s M_1^4}
	\arrow["{p_2\circ d}", two heads, from=1-1, to=1-2]
	\arrow["{p_1\circ d}"', from=1-1, to=2-1]
	\arrow["d", dashed, from=1-2, to=2-2]
	\arrow["-d"', from=2-1, to=2-2]
\end{tikzcd}$$

Hence,
$$\s M_2=
\begin{cases}
\QQ(x,y), & \text{at } \C_{p^2}/e,\\
\QQ(x,y), & \text{at } \C_{p^2}/\C_p,\\
\QQ(x), & \text{at } \C_{p^2}/\C_{p^2},
\end{cases}$$
where $|x|=2$, $|y|=3$, and the differential at $\C_{p^2}/e$ is given by
$
dx=y,$ and $dy=0.
$

Next, let
$$\u V_{3,0}=
\begin{cases}
(z), & \text{at } \C_{p^2}/e,\\
(z), & \text{at } \C_{p^2}/\C_p,\\
(z), & \text{at } \C_{p^2}/\C_{p^2},
\end{cases}
\qquad
\u V_{3,1}=
\begin{cases}
(w), & \text{at } \C_{p^2}/e,\\
0, & \text{at } \C_{p^2}/\C_p,\\
0, & \text{at } \C_{p^2}/\C_{p^2}.
\end{cases}$$

Suppose the differential on $\u V_{3,0}$ is defined by
$dz=x^2+w$
at $\C_{p^2}/e$. Naturality then implies
$dz=x^2$
at both $\C_{p^2}/\C_p$ and $\C_{p^2}/\C_{p^2}$. The differential on $\u V_{3,1}$ is determined by the commutativity of the following diagram:
$$\begin{tikzcd}
	{\u V_{3,0}(\C_{p^2}/e)} &&& {\u V_{3,1}(\C_{p^2}/e)} \\
	& z & w \\
	& {x^2} & {-2xdx=-2xy} \\
	{Z^4(\s M_2)(\C_{p^2}/e)} &&& {\s M_1^4(\C_{p^2}/e)}
	\arrow["{p_2\circ d}", two heads, from=1-1, to=1-4]
	\arrow["{p_1\circ d}"', from=1-1, to=4-1]
	\arrow["d", dashed, from=1-4, to=4-4]
	\arrow[maps to, from=2-2, to=2-3]
	\arrow[maps to, from=2-2, to=3-2]
	\arrow[dashed, maps to, from=2-3, to=3-3]
	\arrow[maps to, from=3-2, to=3-3]
	\arrow["{-d}"', from=4-1, to=4-4]
\end{tikzcd}$$

Hence,
$dw=-2xy$
on $\u V_{3,1}(\C_{p^2}/e)$.
Therefore,
$$\s M_3=
\begin{cases}
\QQ(x,y,z,w), & \text{at } \C_{p^2}/e,\\
\QQ(x,y,z), & \text{at } \C_{p^2}/\C_p,\\
\QQ(x,z), & \text{at } \C_{p^2}/\C_{p^2},
\end{cases}$$
where
$
|x|=2,$ $
|y|=3,$ $
|z|=4,$ $
|w|=5,
$
and the differential at $\C_{p^2}/e$ is given by
$
dx=y,$ $
dy=0,$ $
dz=x^2+w,$ $
dw=-2xy.
$
\end{example}

As seen in the previous example, even with a clear description of minimal systems, it is not easy to find the minimal model for a given diagram of $\DGA$s. We now explore situations in which non-equivariant model data at each stage of a system of differential graded algebras ($\DGA$s) can be combined to construct an equivariant minimal model.

 Given a $\DGA$ $A$ equipped with an action of a group $G$ by $\DGA$-automorphisms, one may regard $A$ as a $\QQ(G)$-module by forgetting its additional algebraic structure; see \cite[Definition 5]{fine1993equivariant}.
 The existence of a $G$-action on the (non-equivariant) minimal model induced by a $G$-action on a $\DGA$ is established in \cite[Theorem 1.3]{GHVP1978}, where a sketch of the proof is provided. We give a proof based on the ideas developed in \cite[Theorem 10.3]{PGJM}.

We also note that the definition of the $n$-th relative cohomology used in \cite{PGJM} is shifted by one degree in our convention, so that their $n$-th relative cohomology corresponds to the $(n-1)$-th relative cohomology in our notation. This convention is consistent with those adopted in \cite{TR1982, LS2002}.

\begin{theorem}
\label{thm: minimal model for each level}
    Let  $A$  be a $\DGA$ equipped with an action of a group $G$, and let $M \xrightarrow{\rho} A$ be its non-equivariant minimal model. Then $M$ naturally inherits a $G$-action such that  $\rho$ becomes a $G$-equivariant map.
\end{theorem}
\begin{proof}
    A non-equivariant minimal model is of the form $M=\cup_{k\geq0} M_k$ with $M_0 = \mathbb{Q}$. Note that $M_0$ has a trivial $G$-action. Assume by induction that  $M_{n-1}$ admits a $G$-action and that the map  $\rho_{n-1}: M_{n-1} \to A$ is $G$-equivariant which is $(n-1)$-isomorphism. We aim to show that this structure extends to $M_n$, and $\rho_n$ becomes an $n$-isomorphism.

    Let $V$ be the relative cohomology group defined by
    $
    V := \H^{n}(A,M_{n-1}).
    $
    We recall that the cohomology is computed using the chain complex $R$ with components $R^i = M_{n-1}^{i+1} \oplus A^i$, and the differential given by:
    $
    d(m, a) := (dm, \rho_{n-1}(m) - da),$ for $(m, a) \in R^i.
    $
    This construction equips both $R$ and $V$ with the structure of $\QQ(G)$-modules via the diagonal action, and $\rho_{n-1}$ is $G$-equivariant. Therefore, the differential is also $G$-equivariant.

    Next, we define
    $
    M_n := M_{n-1} \otimes \mathbb{Q}(V).
    $
    Since $\QQ(G)$-modules are semisimple, the projection
    $
    Z^{n}(A, M_{n-1}) \to \H^{n}(A,M_{n-1})
    $
    admits a $\QQ(G)$-module section
    $
    s: \H^{n}(A,M_{n-1}) \to Z^{n}(A, M_{n-1}).
    $

    For each $v \in V$, let  $s(v) = (m_v, a_v) \in M_{n-1}^{n+1} \oplus A^n$. Then  $dm_v = 0$ and $da_v = \rho_{n-1}(m_v)$. Define the differential and $\rho_n:M_n\to A$ by:
    $
    dv = m_v, $ $ \rho_n(v) = a_v.
    $
    We note that $d, \rho$ are $G$-equivariant since $s$ is $G$-equivariant.
    This defines a $G$-equivariant extension of the minimal algebra $M_{n-1}$. One can show $\H^i(A,M_n)=0$ for all $i\leq n$ by following the non-equivariant arguments given in \cite[Lemma 10.4]{PGJM}. This concludes the proof.
\end{proof}

\begin{remark}\label{rem: insufficiency to be a correct model}
    In \cite{GHVP1978},  minimal models equipped with a group action are referred to as equivariant minimal models. However, these models do not capture the equivariant rational homotopy type of spaces. To illustrate this, we consider $\C_2$-spaces $X$ and $Y$,  both are homeomorphic to $S^5$, but equipped with different involutions as follows:
    \begin{align*}    
    &\xi \cdot (x_1, x_2, x_3, x_4, x_5, x_6) = (-x_1, x_2, x_3, x_4, x_5, x_6)\quad \text{on } X \end{align*}
    and \begin{align*}\xi \cdot (x_1, x_2, x_3, x_4, x_5, x_6) = (-x_1, -x_2, -x_3, x_4, x_5, x_6)\quad \text{on } Y,\end{align*}
     where $S^5$ regarded as the unit sphere in $\RR^6$, and $\xi$ denotes the generator of $\C_2$. 
     
    These two $\C_2$-spaces have different equivariant rational homotopy types. Indeed,
    $X^{\C_2}\cong S^4,$ $ Y^{\C_2}\cong S^2,$
    so their fixed point sets are not even rationally homotopy equivalent.  Nevertheless, both spaces have the same minimal model $(\QQ(a),0)$, where $|a|=5$, and the induced $\C_2$-action is given by $\xi\cdot a=-a.$

     This limitation motivates the theory developed by Triantafillou and Scull \cite{TR1982, LS2002}. Their notion of an equivariant minimal model, recalled in \cref{def:minimal}, completely characterizes equivariant rational homotopy types. Rather than considering a single $\DGA$ equipped with a group action, they consider a diagram of $\DGA$s that encodes the rational homotopy types of all fixed point spaces together with the morphisms induced by the group action.
     Unless stated otherwise, the term equivariant minimal model in this article will always refer to Triantafillou-Scull's notion.
\end{remark}

\begin{proposition}
 \label{prop: lower fixed DGA minimal}
      Let $\sA$ be a $\sO_G$-$\DGA$ 
      such that $\sA(G/H) = \QQ$ for all subgroups $H > e$. Then the minimal model of $\sA$ can be constructed from the non-equivariant minimal model of $\sA(G/e)$.
\end{proposition}
\begin{proof}
  By \cref{thm: minimal model for each level}, the non-equivariant minimal model of $\sA(G/e)$ inherits a $\QQ(G/e)$-action with
    $\rho: M \to \sA(G/e)$
     an equivariant $\DGA$ map. Define $\sM(G/e) = M$ and $\sM(G/H) = (\QQ,0)$ for all $H > e$, with the structure maps defined by augmentations and identity maps. Construct $\sM$  by induction as  $\s M_n=\sM_{n-1}\otimes \QQ(\uV_n)$, where $\sM_0=\u \QQ$.
     
     Here, $\u V_n$ is given by $$\uV_n(G/H)=\begin{cases}
         W_e &\text{ for } H= e
         \\ 0 &\text{ otherwise}. 
     \end{cases}$$
     for $n\geq 2,$ 
     where $ W_e$ is the $\QQ(G)$-module generated by $n$-th degree generators of $\s M$.
     Since  $\u V_n(G/H)=0$  for $H\neq e$, one can identify $\u V_n$ with $\u W_e(\_)=\hom_{\QQ(G)}(\QQ(G)^{(-)},W_e)$ 
     (\cref{defn: standard inj env in vec}). It follows from \cref{sum inj} that $\u V_n$ is injective in $\vc^\ast$. Consequently, the extensions $\s M_n=\sM_{n-1}\otimes \QQ(\uV_n)$ are  elementary extensions. Hence $\s M$ is a minimal system. We extend $\rho$ by setting $\rho(G/H)=\id$ for $H\neq e.$ 
    Therefore, the map $\rho: \sM \to \sA$ provides the minimal model of $\sA$.
\end{proof}

\begin{example}
Consider the system of graded algebras $\s A$ over $\C_{4}$, where $$\s A(\C_{4}/H) = \begin{cases}  \QQ(x_1, x_2, x_3, x_4) / (x_1^2, x_2^2, x_3^2, x_4^2) & H=e  \\   \QQ & H \in \{ \C_2, \C_4 \}.\end{cases} $$ The degrees of the generators are $\vert{}x_i\vert{} = 2$ for $1 \leq i \leq 4$. The action of the generator $\xi \in \C_4$ is given by $\xi \cdot x_i = x_{i+1}$ for $1 \leq i \leq 3$, and $\xi \cdot x_4 = x_1$. 

The minimal model $\s M$ of $\s A$ takes the form $\s M(\C_{4}/e) = \QQ(x_1, x_2, x_3, x_4, y_1, y_2, y_3, y_4),$ equipped with the differential $d(x_i) = 0$ and $d(y_i) = x_i^2$ for $1 \leq i \leq 4$. The generator degrees are $\vert{}x_i\vert{} = 2$ and $\vert{}y_i\vert{} = 3$. The $\C_4$-action extends equivariantly via $\xi \cdot x_i = x_{i+1}$ and $\xi \cdot y_i = y_{i+1}$ for $1 \leq i \leq 3$, alongside $\xi \cdot x_4 = x_1$ and $\xi \cdot y_4 = y_1$. Finally, $\s M(\C_{4}/H) = \QQ$ for $H \in \{ \C_2, \C_4 \}$.\end{example}
A natural question is how the equivariant minimal model is related to the minimal model at each level $\s A(G/H)$. We answer this question for $G=C_{p^k}$. 
\begin{definition}\label{def:splitDGA}
    We call an $\sO_{C_{p^k}}$-$\DGA$ $\sA$ a \emph{split-$\sO_{C_{p^k}}$-$\DGA$} if for all $i \geq 1$,
    $$ \sA(C_{p^k}/p^i) \cong \sA(C_{p^k}/p^{i-1})\oplus \ker \sA(\hat e_{p^i,p^{i-1}}) $$ as DGA.
\end{definition}

\begin{theorem}\label{thm: levelwise minimal}
    Let $\sA$ be a split-$\sO_{C_{p^k}}$-$\DGA$. Then the minimal model of $\sA$ is levelwise minimal.
\end{theorem}
\begin{proof}
    We prove this statement using induction. In this proof, we exploit the linear poset structure of the orbit category $\s O_G$ for the group $G={C_{p^k}}.$  Assume that the minimal model of $\sA(G/p^{i-1})$ has been constructed and is given by a $\QQ(G/p^{i-1})$-$\DGA$ map $\rho_{i-1}$. Using \cref{thm: minimal model for each level}, we construct a $\QQ(G/p^{i})$-minimal model, say $\u N(G/p^i)$, given by the map $\rho'_{i} : \u N(G/p^{i}) \to \ker \sA(\hat e_{p^i,p^{i-1}}) \oplus \QQ$.

    Next, we construct the $\QQ(G/p^i)$-minimal model $\u M(G/p^{i}) \xrightarrow{\sigma_i} \u M(G/p^{i-1}) \oplus_\QQ \u N(G/p^{i})$ by \cref{thm: minimal model for each level}.
    We define $\u M(\hat e_{p^{i},p^{i-1}}) := p \circ \sigma_i$, where $p$ denotes the projection onto the first component. Therefore, we have the following commutative diagram:
    $$\begin{tikzcd}
        {\u M(G/p^i)} & {\u M(G/p^{i-1}) \oplus_\QQ \u N(G/p^{i})} & {\sA(G/p^i) = \sA(G/p^{i-1}) \oplus \ker \sA(\hat e_{p^{i},p^{i-1}})} \\
        {\u M(G/p^{i-1})} & {\u M(G/p^{i-1})} & {\sA(G/p^{i-1})}
        \arrow["{\sigma_i}", from=1-1, to=1-2]
        \arrow["{\u M(\hat e_{p^{i},p^{i-1}})}", from=1-1, to=2-1]
        \arrow["{\rho_{i-1} \oplus_\QQ \rho'_{i}}", from=1-2, to=1-3]
        \arrow["{p}"', from=1-2, to=2-2]
        \arrow["{p}", from=1-3, to=2-3]
        \arrow["\id"', from=2-1, to=2-2]
        \arrow["{\rho_{i-1}}"', from=2-2, to=2-3]
    \end{tikzcd}$$
    We note that the composition $(\rho_{i-1} \oplus_\QQ \rho'_{i}) \circ \sigma_i =: \rho_i$ is a quasi-isomorphism since $\rho_{i-1}$, $\rho'_i$, and $\sigma_i$ are quasi-isomorphisms.
    These fit together to give a quasi-isomorphism of $\sO_G$-$\DGA$s $\rho: \u M \to \sA$, where $\rho(G/p^i) = \rho_i$.

    Now we will show that $\u M$ is the equivariant minimal model of $\sA$. It is sufficient to show that the equivariant minimal model $\sM$ of $\u M$ is $\u M$ itself. We inductively assume $\sM_{n-1} = \u M_{n-1}$. Let $\uV_n := \uH^n(\u M, \sM_{n-1})$; then $\uV_n = \uH^n(\u M, \u M_{n-1})$. Clearly, $\uV_n$ consists of the $n$-th degree generators of $\u M$.

    We note that $\sigma_i$ is surjective. To justify this, suppose the minimal models are given by $\u M(G/p^{i-1}) = (\QQ(U), d)$ and $\u N(G/p^i) = (\QQ(U'), d)$. Then the minimal model of $\u M(G/p^{i-1}) \oplus_\QQ \u N(G/p^{i})$ takes the form $\u M(G/p^i) = \QQ(U, U', U'')$ where $d(U'') \subseteq \QQ^+(U) \wedge \QQ^+(U')$, and the map $\sigma_i$ is the identity on $U$ and $U'$, and zero on $U''$. 
    Thus, the composition $p \circ \sigma_i = \u M(\hat e_{p^{i},p^{i-1}})$ is induced by the projection map $U \oplus U' \oplus U'' \to U$. 
    Therefore, restricting to the subspaces consisting of $n$-th degree elements, the map $\uV_n(\hat e_{p^i, p^{i-1}})$ is surjective. Hence, $\uV_n$ is injective by \cref{thm: injectivity in Z_p^n}. This implies $\sM_n = \sM_{n-1} \otimes \QQ(\uV_n) = \u M_n$.
\end{proof}

We give an example that explicitly illustrates how an $\s O_G$-splitting produces a levelwise minimal system.
\begin{example}
    Consder the $\sO_{C_p}$-$\DGA$ $$\sA= \begin{cases}

    \QQ(x, y, z)/(  xz, yz) , &\text{at}\ \C_{p}/e\\  
        \QQ(x, y), &\text{at}\ \C_{p}/\C_{p}
\end{cases}$$
The action on each level is trivial and $dy=x^2, dz=0$, and $|x|=2, |z|=3.$ The structure map is $\sA(\hat e_{e,\C_p}):x \mapsto x, y\mapsto y$.
Direct computations show the $4$-th step of the equivariant minimal model is 
$$\sM_4= \begin{cases}
    \QQ(x, y, z,w) , &\text{at}\ \C_{p}/e\\  
        \QQ(x, y), &\text{at}\ \C_{p}/\C_{p}.
\end{cases}$$
where the description of the generators is the same as in $\s A $, and $dw=xz$.
\end{example}

The following is an example of a system of cohomology system $\s A$  that does not split, and its  equivariant minimal model is not levelwise minimal. %This suggests that the splitting hypothesis of \cref{thm: levelwise minimal} can not be weakened. 
\begin{example}\label{ex: minimal system but not levelwise minimal}
    Define the $\sO_{\C_p}$-$\DGA$ $$\sA= \begin{cases}
    \QQ(x)/( x^3) , &\text{at}\ \C_p/e\\  
         \QQ(x)/( x^2), &\text{at}\ \C_p/\C_p
\end{cases}$$
where all the differentials are zero, and $|x|=2$. The structure map is $\sA(\hat e_{e,\C_p}):x \mapsto x$. Thus, $\sA$ is injective.

We construct the minimal model iteratively by setting $\sM_n=\sM_{n-1}(\uV_n)$, where $\uV_n=\uH^{n}(\s A,\sM_{n-1})$, equipped  with a suitable $d|_{\uV_n}$, and an $n$-isomorphism $\rho_n:\sM_n \ra \sA.$
At stage $n=2$, we have $$\uV_2=\begin{cases}
    ( x) , &\text{at}\ \C_p/e\\  
         ( x), &\text{at}\ \C_p/\C_p
\end{cases}$$
We have $$\sM_2=\begin{cases}
    \QQ(x_1) , &\text{at}\ \C_p/e\\  
          \QQ(x_1), &\text{at}\ \C_p/\C_p
\end{cases}$$
with $dx_1=0, |x_1|=2$, and $\rho (x_1)=x.$
Next, we have $$\uV_3=\begin{cases}
    0 , &\text{at}\ \C_p/e\\  
         ( x_1^2), &\text{at}\ \C_p/\C_p
\end{cases}$$
As $\uV_3$ is not injective, we take its injective resolution namely $\uV_3 \xhookrightarrow{} \uV_{3,0} \ra \uV_{3,1}\ra 0.$ Renaming the generator $x_1^2$ by $x_2$, we have the following diagram:
$$\begin{tikzcd}
	0 & {( x_2 )} & {( x'_2 )} \\
	{( x_2 )} & {( x_2 )} & 0
	\arrow[hook, from=1-1, to=1-2]
	\arrow[from=1-1, to=2-1]
	\arrow["x_2\mapsto x'_2", from=1-2, to=1-3]
	\arrow[from=1-2, to=2-2]
	\arrow[from=1-3, to=2-3]
	\arrow["\id", hook, from=2-1, to=2-2]
	\arrow[from=2-2, to=2-3]
\end{tikzcd}$$
Then, $$\sM_3=\begin{cases}
    \QQ(x_1,x_2,x_2') , &\text{at}\ \C_p/e\\  
         \ \QQ(x_1,x_2), &\text{at}\ \C_p/\C_p
\end{cases}$$
with differentials $dx_2=x'_2+x_1^2, dx'_2=0$ at $\C_p/e$, and $dx_2=x_1^2$ at $\C_p/\C_p$, where the degrees $|x_2|=3$ and $|x'_2|=4$. Also, $\rho(x_2)=0$ and $\rho(x'_2)=-x^2$.
Since the corresponding $4$-th relative cohomology vanishes, we get $\sM_4=\sM_3$.

Finally, $$\uV_5=\begin{cases}
   ( x_1^3) , &\text{at}\ \C_p/e\\  
         0, &\text{at}\ \C_p/\C_p
\end{cases}$$
Thus, 
$$\sM_5=\begin{cases}
    \QQ(x_1,x_2,x_2',x_3) , &\text{at}\ \C_p/e\\  
         \ \QQ(x_1,x_2), &\text{at}\ \C_p/\C_p
\end{cases}$$
with $dx_3=x_1^3$, and $\rho(\C_p/e)(x_3)=0.$ This completes the construction of the minimal model as $\s M_5$ is quasi-isomorphic to $\s A$. It is clear that $\sM(\C_p/e)$ is not a minimal algebra.
\end{example}

The converse of \cref{thm: levelwise minimal} is not true, as seen in the example given below.
\begin{example}
    Given a cohomology diagram $\sA$ with minimal model levelwise minimal, it is not necessarily true that $\sA$ is a split-$\sO_G$-$\DGA$. 
    Consider the $\sO_{C_p}$-diagram of graded algebra given by $$\sA=\begin{cases}
     \QQ(x_3,y_3),&\text{at}\ \C_{p}/e\\
        \QQ(z_3), &\text{at}\ \C_{p}/\C_{p}
    \end{cases}$$
    with the structure map $x_3\mapsto z_3$ and $y_3\mapsto 0$. Note that this diagram of graded algebras does not split (see \Cref{def:splitDGA}). It can be verified that $\sA$ is a minimal system,  which is also levelwise minimal. Therefore, splitting is a sufficient but not a necessary condition. 
\end{example}

The following proposition characterizes levelwise minimality of a minimal system in terms of the projectivity of the diagram of the corresponding homotopy groups in $\vc$.
\begin{proposition}\label{prop: proj homotopy implies levelwise}
  Let $G$ be a finite group. The diagram of homotopy groups of a $G$-space $X$ is projective in $\vc$ if and only if the minimal model of $X$ is levelwise minimal.
\end{proposition} 
\begin{proof}
    Let $\sM$ be the minimal model of $X$. 
    Then $\sM=\cup \sM_n$, where $\sM_n=\sM_{n-1}\otimes \mathbb{Q}(\underline{V_n})$, is levelwise minimal if and only if $\uV_n$ is injective in $\vc^\ast$. Indeed, for each $n$, if $\uV_n$ is injective, then during the $n$-th stage of the elementary extension we are adding only generators of degree $n$ at each $G/H$-level, and hence each $G/H$-level of $\sM_n$ is non-equivariantly minimal. On the other hand, if some $\uV_n$ is not injective, then one must take an injective resolution that gives new generators of degree $\geq n+1$ at some $G/H$-level, and then in $\sM_n(G/H)$ the differential is indecomposable, and hence it's not minimal.
    
   We note that $\u A\in \vc^\ast$ is injective if and only if the dual rational coefficient system $\u A^\ast$ defined by $A^\ast (G/H)= \hom (A(G/H),\QQ)$ is projective in $\vc$. Indeed, this can be observed by the equivalence of the following liftings:
$$\begin{tikzcd}
	0 &&& 0 \\
	\u M & \u A & \Leftrightarrow & {\u M^\ast} & {\u A^\ast} \\
	\u N &&& {\u N^\ast}
	\arrow[from=1-1, to=2-1]
	\arrow[from=2-1, to=2-2]
	\arrow[from=2-1, to=3-1]
	\arrow[from=2-4, to=1-4]
	\arrow[from=2-5, to=2-4]
	\arrow[dashed, from=2-5, to=3-4]
	\arrow[dashed, from=3-1, to=2-2]
	\arrow[from=3-4, to=2-4]
\end{tikzcd}$$
    We recall the equivariant Postnikov decomposition of the space $X$, which is contructed by principle fibrations $X^H_n \ra X^H_{n-1}$ whose fiber is $K( \pi_n(X^H),n)$. The duality between a space and its corresponding minimal model gives $\underline{V}_n=\u \pi_n(X)^\ast$ by \cite[Theorem 4.16]{LS2002}.
    This concludes the proof.
\end{proof}
\begin{example}
Recall that a split-$\s O_{\C_{p^k}}$-DGA is injective in $\vc^\ast$. Moreover, the equivariant minimal model of such a system is levelwise minimal (\cref{thm: levelwise minimal}).
    Therefore, by \cref{prop: proj homotopy implies levelwise}, a $G$-space $X$ for which  $\u\sE(X)$ is weakly equivalent to a
    split-$\s O_{\C_{p^k}}$-$\DGA$ has a projective diagram of homotopy groups in $\vc$. 
\end{example}

\section{Equivariant formality of \texorpdfstring{$G$}{G}-spaces}\label{sec: equivariant formality}

For any $G$-space, the minimal model for its corresponding diagram of PL de Rham algebras will give its equivariant rational homotopy type. This is easier to compute when the space is equivariantly formal. In this section, we give examples of equivariantly formal spaces and explore which space-level constructions preserve formality, contrasting our results with the non-equivariant case.  

Recall that a cohomologically $1$-connected $\DGA$ is said to be \emph{formal} if its minimal model is isomorphic to the minimal model of its cohomology algebra. Equivalently, a $\DGA$ $A$, is formal if $A$ and $\H^\ast (A)$ are weakly equivalent. 
A simply connected space is said to be \emph{formal} if the corresponding PL de  Rham algebra is formal. The rational homotopy groups of formal spaces can be computed from their cohomology algebra, and their rational Postnikov tower can also be computed (\cite[Theorem 3.3]{DGMS1975}).

The following proposition relates a weak equivalence between the PL de Rham system of a $G$-space and its corresponding cohomology system to a quasi-isomorphism from the minimal model of the $G$-space to an injective replacement of the corresponding cohomology system. The underlying idea of the proof is implicitly used in \cite{fine1993equivariant}. However,  the proposition is of independent interest and leads to several important observations.
\begin{proposition}\label{prop: formal defn}
    Let $X$ be a $ G$-space and $\sM_X$ be the minimal model of $X$. Then the following are equivalent:
    \begin{enumerate}
        \item There exists a morphism of systems of $\DGA$s, 
$\sM_X\longrightarrow \sI(\uH^\ast(X))
$
which is a quasi-isomorphism.
\item There exists a weak equivalence between the cohomology system $\uH^\ast (X)$ and the PL de Rham system $\u{\s E}_X.$ 
\end{enumerate}
\end{proposition}
\begin{proof}
    We first prove $(1)\implies (2)$. Assume we have quasi-isomorphisms $\underline{\s E}(X) \xleftarrow{\sim} \sM_X \xra{\sim} \sI(\uH^\ast (X))$. By taking the inclusion $\uH^\ast (X)\ra \sI(\uH^\ast (X))$ as in \Cref{thm:fine inj extension}, we get a zig-zag of quasi-isomorphisms $\underline{\s E}(X) \xleftarrow{\sim} \sM_X \xra{\sim} \sI(\uH^\ast (X)) \xleftarrow{\sim}\uH^\ast (X)$. Thus $\uH^\ast (X)$ is weakly equivalent to $\u{\s E}_X.$

     Next we show $(2)\implies (1)$. Assume there is a weak equivalence between $\uH^\ast (X)$ and  $\u{\s E}_X$. Suppose the weak equivalence is given by the zig-zag of quasi-isomorphisms $$\uH^\ast (X)\xleftarrow{\sim} \s A_1\xra{\sim} \cdots\xleftarrow{\sim} \s A_n\xra{\sim}  \u{\s E}_X.$$
     By \cref{thm:fine inj extension}, we have the vertical inclusions, which are quasi-isomorphisms. Using \cref{prop: fine factor through inj}, we obtain extensions given by the lower dashed  horizontal arrows in the following diagram:
$$\begin{tikzcd}
	{\uH^\ast (X)} & {\s A_1} & \cdots & {\s A_n} & { \u{\s E}_X} \\
	{\s I(\uH^\ast (X))} & {\s I(\s A_1)} & \cdots & {\s I(\s A_n)} & { \u{\s E}_X}
	\arrow["\sim", from=1-1, to=2-1]
	\arrow["\sim"', from=1-2, to=1-1]
	\arrow["\sim", from=1-2, to=1-3]
	\arrow["\sim", from=1-2, to=2-2]
	\arrow["\sim"', from=1-4, to=1-3]
	\arrow["\sim", from=1-4, to=1-5]
	\arrow["\sim", from=1-4, to=2-4]
	\arrow["\id", from=1-5, to=2-5]
	\arrow[dashed, from=2-2, to=2-1]
	\arrow[dashed, from=2-2, to=2-3]
	\arrow[dashed, from=2-4, to=2-3]
	\arrow[dashed, from=2-4, to=2-5]
\end{tikzcd}$$
Notice that we do not require an injective envelope of $\u{\s E}_X$ for the extension, since  $\u{\s E}(X)$ is injective by \cite[Proposition 4.3]{TR1982}.
     By the two-out-of-three property for quasi-isomorphisms, lower horizontal arrows are quasi-isomorphisms. 
     Applying \cref{prop: lifting property}, we obtain the following homotopy-commutative diagram where the vertical arrows represent minimal models (note the injectivity ensures the existence of minimal models):
$$\begin{tikzcd}
	{\s M^{(0)}} & {\s M^{(1)}} & \cdots & {\s M^{(n)}} & {\s M} \\
	{\s I(\uH^\ast (X))} & {\s I(\s A_1)} & \cdots & {\s I(\s A_n)} & { \u{\s E}_X}
	\arrow["\sim", from=1-1, to=2-1]
	\arrow[dashed, from=1-2, to=1-1]
	\arrow[dashed, from=1-2, to=1-3]
	\arrow["\sim", from=1-2, to=2-2]
	\arrow[dashed, from=1-4, to=1-3]
	\arrow[dashed, from=1-4, to=1-5]
	\arrow["\sim", from=1-4, to=2-4]
	\arrow["\sim", from=1-5, to=2-5]
	\arrow["\sim"', from=2-2, to=2-1]
	\arrow["\sim", from=2-2, to=2-3]
	\arrow["\sim"', from=2-4, to=2-3]
	\arrow["\sim", from=2-4, to=2-5]
\end{tikzcd}$$
The dashed arrows are quasi-isomorphisms by the two-out-of-three property.  Therefore, by \cref{thm: uniqueness of minimal alg},  all the morphisms represented by the upper arrows become isomorphisms (possibly after a modification up to homotopy).
Thus, we obtain  the required quasi-isomorphisms $\underline{\s E}(X) \xleftarrow{\sim} \sM \xra{\sim} \sI(\uH^\ast (X))$.
\end{proof}
Similarly, we have the following equivalent statements for a system of $\DGA$s.
\begin{proposition}    
If $\sA$ is a system of $\DGA$s  and $\sM_{\s A}$ be the minimal model of $\s A$. Then the following are equivalent: 
\begin{enumerate}
     \item There exists a morphism of systems of $\DGA$s 
$\sM_{{\sA}}\longrightarrow \sI(\uH^\ast(\sA))
$
which is a quasi-isomorphism.
\item There exists a weak equivalence between the cohomology system $\uH^\ast (\sA)$ and the minimal model $\sM_{\sA}$. 

\end{enumerate}
\end{proposition}

\begin{definition}
    A $G$-space $X$ is said to be \emph{equivariantly formal} if it satisfies one of the equivalent conditions in \Cref{prop: formal defn}.
Similarly, a system of $\DGA$s $\sA$ is said to be \emph{equivariantly formal} if it satisfies one of the equivalent conditions in \Cref{prop: formal defn}.
\end{definition}
\begin{remark}
    Note that in \cite{fine1993equivariant} the authors use part $(1)$ of \Cref{prop: formal defn} in order to define equivariant formal spaces.
\end{remark}

\begin{remark}
    A different notion  of equivariant formality,
introduced in \cite{GKM1998} and frequently occurs in the literature, is related to the existence of equivariant extensions of an element in ordinary cohomology. More precisely, a $G$-space is called \emph{equivariantly formal} if its Borel cohomology satisfies $\H_G^\ast(X)\cong\H^\ast(BG)\otimes \H^\ast(X)$ as $\H^\ast(BG)$-modules, or equivalently, the Serre spectral sequence associated to the fibration $X\to X\times _GEG\to BG$ collapses at the $E_2$-page. We emphasize that this notion of equivariant formality is distinct from the one considered in this paper, and we do not address it here.
\end{remark}

\begin{corollary}\label{rem: fixed point set of eq formal is formal}    
    If  a $G$-space $X$ is equivariantly formal, the equivalence between $\u \sE(X)$ and $\uH^\ast(X)$  yields, in particular, a  weak equivalence between $\sE(X^H)$ and $\H^\ast(X^H)$ for all $H\leq G$. Therefore, all $H$-fixed point spaces $X^H$ are non-equivariantly formal. 
In particular, for a trivial $G$-action, the equivariant formality coincides with the non-equivariant formality.
\end{corollary}

\begin{proposition}
    
\label{cor: restriction preserves formality}
Let $X$ be a $G$-space and let $K\leq G$. If $X$ is equivariantly formal as a $G$-space, then $X$, regarded as a $K$-space by restriction of the action, is equivariantly formal.
\end{proposition}
\begin{proof}
Since $X$ is equivariantly formal as a $G$-space, there exists a weak equivalence (due to \Cref{prop: formal defn}) i.e., zig-zag of quasi-isomorphisms of $\mathscr O_G$-$\DGA$s
$
\uH^\ast(X)
\to\cdots\xleftarrow{}
\u{\s E}_X.
$
The inclusion $K\hookrightarrow G$ induces a restriction functor
$
\mathrm{Res}^G_K:\mathscr O_G\text{-DGA}
\longrightarrow
\mathscr O_K\text{-DGA},
$
given by precomposition with the canonical functor
$
\mathscr O_K\longrightarrow \mathscr O_G,
$ $
K/L\longmapsto G/L .
$
Applying $\mathrm{Res}^G_K$ to the above zig-zag yields a zig-zag of quasi-isomorphisms of $\mathscr O_K$-$\DGA$s
$
\mathrm{Res}^G_K\bigl(\uH^\ast_G(X)\bigr)
\to\cdots\xleftarrow{}
\mathrm{Res}^G_K\bigl(\u{\s E}_X\bigr).
$
Since quasi-isomorphisms for $\s O_G$-$\DGA$s are defined level-wise, the restriction functor preserves quasi-isomorphisms. 
Therefore, $X$ is equivariantly formal as a $K$-space.
\end{proof}
We study the behaviour of equivariant formality under natural topological constructions, including products, wedges, and retractions, by analyzing the corresponding equivariant minimal models. While some of these constructions preserve equivariant formality, many constructions known to be formal non-equivariantly do not generalize to equivariant spaces. We provide counterexamples that illustrate the failure of several important results in the equivariant setting (see \cref{ex: join is not eqformal} and the subsequent remarks).

\subsection{Products.} First, we consider products in the equivariant setup. Suppose $G$ acts on $X_i$ for $i=1,\dots, m$. This induces a canonical diagonal $G$-action on $\prod_{i=1}^m X_i$ defined by $$g\cdot(x_1,\dots,x_m)=(gx_1,\dots, gx_m).$$ 
\begin{theorem}
\label{prop:product of equivariantly formal spaces is formal}
    Let $G$ be a finite group acting on spaces $X_i$ for $i=1,\dots, m$, and let $\sM_i$ denote the minimal models of $X_i$. Consider the $G$-space $\prod_{i=1}^m X_i$ equipped with the diagonal $G$-action.
    \begin{enumerate}
        \item 
     Then, the tensor product $\bigotimes_{i=1}^m\s M_i$ is the minimal model for $\prod_{i=1}^m X$. 
    \item Moreover, if each $X_i$ is  equivariantly formal, then the product $\prod_{i=1}^m X_i$ is also equivariantly formal.
    \end{enumerate}
\end{theorem}\begin{proof}
It is enough to prove the proposition for two spaces, say $X$ and $Y$. The result for finite products then follows by a finite iteration.

\textbf{1.}    Let us denote the projection maps onto the components by $\pi_X: X \times Y \ra X$ and $\pi_Y: X \times Y \ra Y$.  Both are clearly $G$-maps,  since $X\times Y$ is equipped with the diagonal $G$-action. Applying the equivariant PL de Rham functor, we obtain morphisms of $\DGA$ diagrams $\u{\s E}(\pi_X): \u{\s E}(X)\ra  \u{\s E}(X\times Y)$ and $\u{\s E}(\pi_Y): \u{\s E}(Y)\ra  \u{\s E}(X\times Y)$ respectively. Now, the map
    $$\u{\s E}(\pi_X)\cdot \u{\s E}(\pi_Y) :  \u{\s E}(X)\otimes \u{\s E}(Y) \ra  \u{\s E}(X\times Y)$$ 
    gives a quasi-isomorphism of $\sO_G$-$\DGA$s, where $\u{\s E}(X)\otimes \u{\s E}(Y)$ is regarded as an $\sO_G$-$\DGA$ with $G$ acting componentwise.
    A non-equivariant argument for this quasi-isomorphism can be found in \cite[Corollary 10.10, p. 142: Example 2]{FHT}. Therefore, we may use $\u{\s E}(X)\otimes \u{\s E}(Y)$ as a model for the PL de   Rham system of the product in place of $\u{\s E}(X\times Y)$.
    
   Consider quasi-isomorphisms $\rho: \sM \to \underline{\mathscr E}{(X)}$ and $\rho': \sN \to \underline{\mathscr E}(Y)$. Applying the Künneth formula, we obtain a quasi-isomorphism of $\sO_G$-$\DGA$s: 
   $
 \rho \otimes \rho': \sM\otimes \s N \to \underline{\mathscr E}{(X)}\otimes  \underline{\mathscr E}{(Y)}.
   $
   We only need to show that $\sM\otimes \s N$ is a minimal system, that is, obtained by iterative elementary extensions.

   Recall that the $n$-th stages of the minimal models $\sM$ and $\s N$ are obtained by taking elementary extensions of $\s M_{n-1}$ and $\s N_{n-1}$ with respect to some maps $\gamma: \u V_n \to \u Z^{n+1}(\sM_{n-1})$ and $\gamma': \u V'_n \to \u Z^{n+1}(\sN_{n-1})$, respectively. Taking injective resolutions of $\u V_n$ and $\u V'_n$, the $n$-th truncated part of the minimal models  take the following forms:
   $
   \sM_{n}=\s M_{n-1} \otimes \QQ(\uV_{{n},0} \oplus \cdots\oplus  \uV_{n,m} ),
   $ and 
   $
   \sN_{n}=\s N_{n-1} \otimes \QQ(\u V'_{n,0}\oplus \cdots\oplus   \uV'_{n,m'} )
   $
   for some integers $m, m'$.
   Therefore, 
  $$ \sM_{n}\otimes \s N_n=(\s M_{n-1} \otimes \s N_{n-1})\otimes \QQ((\uV_{{n},0}\oplus \u V'_{n,0}) \oplus (\uV_{{n},1}\oplus \u V'_{n,1}) \oplus \cdots).$$   
   We give an inductive argument to show that $\sM \otimes \s N$ is obtained by elementary extensions. We start with $\sM_0=\u \QQ$ and $\s N_0=\u \QQ$, so that $\sM_0 \otimes \s N_0=\u \QQ$, and proceed by induction on $n$.
   We inductively assume that the $(n-1)$-th stage, $\sM_{n-1}\otimes \s N_{n-1}$, is a minimal system.

   We define
   $
   \u W_n :=\u V_n \oplus \u V'_n. 
   $ 
   Then, we have a morphism in the category $\vc^\ast$ given by 
  $$ \gamma\oplus \gamma': \u W_n \to  Z^{n+1}(\sM_{n-1})\oplus  Z^{n+1}(\sN_{n-1}) \subseteq  Z^{n+1}(\sM_{n-1} \otimes \sN_{n-1}).$$
   From the decomposition of $\s M_n \otimes \s N_n$, it is clear that it is obtained by an elementary extension via $\gamma\oplus \gamma'$:
   $
    \sM_n\otimes \s N_n = (\sM_{n-1}\otimes \s N_{n-1})(\u W_n),
   $
   which implies that $\sM_n\otimes \s N_n$ is a minimal system.
   Thus, $\s M\otimes \s N=\colim_n (\sM_n\otimes \s N_n)$ is a minimal system.

  \textbf{2.}  Finally, we establish the equivariant formality of the product. Suppose there are quasi-isomorphisms $\sigma:\sM \to \s I(\uH^\ast(X))$ and $\sigma':\sN \to \s I(\uH^\ast(Y))$, since $X$ and $Y$ are equivariantly formal. Tensoring these together, we obtain a quasi-isomorphism 
  $$ \sigma\otimes \sigma':\s M\otimes \s N \to  \s I(\uH^\ast(X)) \otimes \s I(\uH^\ast(Y)).$$
   
   This yields a weak equivalence given by the following zig-zag of quasi-isomorphisms:
    $$\underline{\mathscr E}{(X)}\otimes \underline{\mathscr E}{(Y)}  \xleftarrow[]{\rho\otimes\rho'}\s M\otimes \s N \xra{\sigma\otimes \sigma'}  \s I(\uH^\ast(X)) \otimes \s I(\uH^\ast(Y)) \xleftarrow[]{} \uH^\ast(X) \otimes \uH^\ast(Y)\cong\uH^\ast(X\times Y).$$
   The third map is obtained by tensoring the inclusions into the injective envelopes, $\uH^\ast(X)\to\s I(\uH^\ast(X))$ and $\uH^\ast(Y)\to\s I(\uH^\ast(Y))$, which is clearly a quasi-isomorphism.
   By \cref{prop: formal defn}, it follows that $X\times Y$ is equivariantly formal.
\end{proof}

We now turn to a different class of $G$-spaces arising naturally from topological spaces. Specifically, we consider the cyclic permutation action of $G=\C_{p^k}$ on the $p^k$-fold product $X^{p^k}.$
Let $\xi$ denote the generator of $G = \C_{p^k}$. 
The action is given by
    $
        \xi \cdot (x_1, x_2, \dots, x_{p^k}) = (x_2, x_3, \dots, x_{p^k}, x_1).
    $

For this space, the equivariant minimal model can be explicitly constructed, and the equivariant formality of the $G$-space can be completely determined by the non-equivariant formality of the underlying space.
\begin{theorem}\label{permutation formal}
 Let $G = \C_{p^k}$ act on the $p^k$-fold product $X^{p^k}$ by cyclic permutation of the factors. Then $X^{p^k}$ is equivariantly formal if and only if $X$ is non-equivariantly formal.
\end{theorem}
\begin{proof}
Let us denote the product as $Y = X^{p^k}$.
    For each $i = 0, \dots, k$, the fixed point space $Y^{(p^i)}$ is homeomorphic to $X^{p^i}$, where $(p^i)$ denotes the subgroup of $\C_{p^k}$ of order $p^{k-i}$. For each $i\geq 1$, the inclusion $Y^{(p^{i-1})} \hookrightarrow Y^{(p^{i})}$ is equivalent to the diagonal map
       $$ X^{p^{i-1}} \xrightarrow{\Delta} X^{p^{i-1}} \times \cdots \times X^{p^{i-1}} \cong X^{p^{i}}.$$
    
    Therefore, the induced map on cohomology
    $
       \H^\ast(Y^{(p^{i-1})}) \xrightarrow{\Delta^\ast} \H^\ast(Y^{(p^i)})
    $
    can be identified with the multiplication map $\mu: \H^\ast(X)^{\otimes p^i}\ra \H^\ast(X)^{\otimes p^{i-1}}$, given by $a_1 \otimes \cdots \otimes a_p \mapsto a_1 \cdots a_p$, where $a_j \in \H^\ast(X)^{\otimes p^{i-1}}.$ Clearly, this map is surjective.
    
    By the definition of the cohomology system, $\uH^\ast (Y)(G/H)=\H^\ast(Y^H)$, and the structure maps $\uH^\ast(Y) (\hat a):\uH^\ast (Y)(G/H)\to \uH^\ast (Y)(G/H')$ for $\hat a:G/H\to G/H'$ are obtained by the inclusion $Y^{H'}\xra{a \cdot} Y^H.$ In particular, the structure map $\uH^\ast(Y)(\hat e_{(p^i),(p^{i-1})})$ is exactly $\Delta^\ast.$
    Therefore, the injectivity of $\uH^\ast(Y)$ in $\vc^\ast$ follows from \cref{thm: injectivity in Z_p^n}.
    
    Note that the projection $\pi_j:X^{p^{i}}=X^{p^{i-1}}\times \cdots\times X^{p^{i-1}}\ra X^{p^{i-1}}$ onto the $j$-th factor induces a $\DGA$ map: 
    $ 
    \sE(\pi_j): \mathscr{E}(X^{p^{i-1}}) \to \mathscr{E}{(X^{p^i})},
    $
    by the functoriality of the PL de Rham algebra for $j=1,\dots, p$. These maps together produce a quasi-isomorphism, given by the multiplication map  
    \begin{align}\label{eq: quasi}
      \prod_{j=1}^{p}  \sE(\pi_j): \mathscr{E}(X^{p^{i-1}})^{\otimes p} \to \mathscr{E}{(X^{p^i})},
    \end{align}
    defined by $\left(\prod_{j=1}^p  \sE(\pi_j)\right)(a_1 \otimes \dots \otimes a_{p}):=\sE(\pi_1)(a_1)\cdots \sE(\pi_p)(a_{p})$.

    We obtain a commutative diagram
$$\begin{tikzcd}
	{\underline{\mathscr{E}}(Y)(G/p^{i-1})} && {\underline{\mathscr{E}}(Y)(G/p^{i-1})} \\
	{\mathscr{E}{(X^{p^i})}} && {\mathscr{E}{(X^{p^{i-1}})}} \\
	{\mathscr{E}{(X^{p^{i-1}})}^{\otimes p}} && {\mathscr{E}{(X^{p^{i-1}})}}
	\arrow["{\underline{\mathscr E}(Y)(\hat e_{p^i,p^{i-1}})}", from=1-1, to=1-3]
	\arrow[shift left, no head, from=1-1, to=2-1]
	\arrow[shift right, no head, from=1-1, to=2-1]
	\arrow[shift right, no head, from=1-3, to=2-3]
	\arrow[shift left, no head, from=1-3, to=2-3]
	\arrow["{\mathscr E(\Delta)}", from=2-1, to=2-3]
	\arrow["{\prod_{j=1}^p\mathscr{E}(\pi_j)}", from=3-1, to=2-1]
	\arrow["\mu"', from=3-1, to=3-3]
	\arrow["\id"', from=3-3, to=2-3]
\end{tikzcd}$$
    since $(\s E(\Delta)\circ \prod_{j=1}^p\s E(\pi_j))(a_1\otimes\cdots \otimes a_p)=\sE(\pi_1\circ \Delta)(a_1)\cdots \sE(\pi_p\circ \Delta)(a_{p})=a_1a_2\cdots a_p$, because $\pi_j\circ \Delta=\id$ for all $j.$
Therefore, the level-wise PL de Rham system $\u{\s E}(Y)$ is quasi-isomorphic to a system ${\s F}$, where $\s F(G/p^i)=\s E(X)^{\otimes p^i}$ with the structure maps $\s F(\hat e_{p^i,p^{i-1}})$ given by the multiplication map $\mu$ as described in the above diagram.

Next, we need to identify $\mathscr F(X^{p^i})(\hat \xi)$. The map corresponding to $\hat \xi:G/{p^i}\to G/p^i$ is given by the cyclic permutation $\sigma:X^{p^i}\to X^{p^i}$ defined by $\sigma(x_1,x_2, \dots,x_{p^i})=(x_2,\cdots,x_{p^i},x_1)$. Therefore, $\u{\s E}(Y)(\hat \xi)=\s E(\sigma).$
To ensure that ${\s F}\to\u {\s E}(X)$ is a map of $\s O_G$-$\DGA$s, we require the following diagram to commute:
$$\begin{tikzcd}
	{\mathscr{E}{(X^{p^i})}} && {\mathscr{E}{(X^{p^{i}})}} \\
	{\mathscr{E}{(X)}^{\otimes {p^i}}} && {\mathscr{E}{(X)}^{\otimes {p^i}}}
	\arrow["{\u{\s E}(Y)(\hat \xi)}", from=1-1, to=1-3]
	\arrow["{\prod_{j=1}^{p^i}\mathscr{E}(\pi_j)}", from=2-1, to=1-1]
	\arrow["{{\s F}(\hat \xi)}"', from=2-1, to=2-3]
	\arrow["{\prod_{j=1}^{p^i}\mathscr{E}(\pi_j)}"', from=2-3, to=1-3]
\end{tikzcd}$$
We note that
    \begin{align*}
    \left(\u{\s E}(Y)(\hat \xi)\circ\prod_{j=1}^{p^i}\s E(\pi_j)\right)\Big(\bigotimes_{j=1}^{p^i}\alpha_j\Big)
    &=\s E(\sigma)(\s E(\pi_1)(\alpha_1)\cdots\s E(\pi_{p^i})(\alpha_{p^i}))\\
    &=\s E(\pi_1\circ\sigma)(\alpha_1)\cdots\s E(\pi_{p^i}\circ \sigma)(\alpha_{p^i})\\
    &=\s E(\pi_2)(\alpha_1)\cdots \s E(\pi_{p^i})(\alpha_{{p^i}-1})\cdot\s E(\pi_1)(\alpha_{p^i})\\
    &=(-1)^{|\alpha_{p^i}|(\sum_{j\neq p^i}|\alpha_j|)} \s E(\pi_1)(\alpha_{p^i})\s E(\pi_2)(\alpha_1)\cdots \s E(\pi_{p^i})(\alpha_{{p^i}-1}),
    \end{align*}
where the third equality follows from the fact that $\pi_j\circ \sigma=\pi_{j+1}$ for $j=1,\dots, p^{i}-1$ and $\pi_{p^i}\circ \sigma=\pi_{1}$.
Thus, by setting 
$${\s F}(\hat \xi)(\alpha_1\otimes \cdots\otimes \alpha_{p^i})=(-1)^{|\alpha_{p^i}|(|\alpha_1|+\cdots+|\alpha_{p^i-1}|)} \alpha_{p^i} \otimes \alpha_1 \otimes \cdots \otimes \alpha_{{p^i}-1},$$
the requirement is fulfilled. We thus obtain a quasi-isomorphism ${\s F}\to\u {\s E}(X).$

Let $\rho:M\to \mathscr{E}(X)$ be the non-equivariant minimal model for $X$. We produce an equivariant minimal model $\sM$ for $\u{\s F}$ using the model $M$. 
    We define $\sM$ by:
    $
        \sM(G/p^i) = M^{\otimes p^i},
    $
    and define the action of the generator $\xi$ by
        $$\sM(\hat \xi)(\alpha_1 \otimes \cdots \otimes \alpha_{p^i}) =(-1)^{|\alpha_{p^i}|(|\alpha_1|+\cdots+|\alpha_{p^i-1}|)} \alpha_{p^i} \otimes \alpha_1 \otimes \cdots \otimes \alpha_{{p^i}-1}.$$

    The structure maps $\sM(\hat{e}_{p^i,p^{i-1}})$ are given by the multiplication maps. Since the injectivity of $\sM$ follows from the same reasoning used for the injectivity of $\uH^\ast(Y)$, $\sM$ is indeed a valid system of $\DGA$s.
    By the construction of $\s M$, we get a quasi-isomorphism $\s M\to \s F.$
    Thus, $\s M$ is a model of $Y$, given by the composition:
$\s M\to\s F\to \u {\s E}(Y).$

Next, we prove that $\sM$ is minimal by showing that $\sM$ can be obtained via an increasing sequence of elementary extensions.
Since $M$ is non-equivariantly minimal, it can be written as $M=\QQ (V_2\oplus \cdots \oplus V_j\oplus \cdots)$, where the $V_j$ are vector spaces concentrated in homogeneous degree $j$. We define $\uV_n$ by $\uV_n(G/p^i)=  V_n^{\oplus p^i}$ and set $\uV_n(\hat e_{p^i,p^{i-1}}): \bigoplus_{r=1}^p(V_n^{\oplus p^{i-1}}) \ra V_n^{\oplus p^{i-1}}$ to be the addition map
$$\uV_n(\hat e_ {p^i,p^{i-1}})(v_1, \dots, v_p) =v_1 + \dots+ v_p.$$
Clearly, $\uV_n$ is injective by \cref{thm: injectivity in Z_p^n}. Inductively, we construct $\sM_n$ as the elementary extension $\sM_n:=\sM_{n-1}(\uV_n)=\s M_{n-1}\otimes  \QQ(\u V_n)$. Since $\s M=\colim_n \s M_n$, $\s M$ is a minimal system. Now, assume $X$ is formal. We then have a non-equivariant quasi-isomorphism $M \to \H^\ast(X)$. Since the cohomology system $\uH^\ast(Y)$ is injective, we can follow the same techniques as above to obtain a quasi-isomorphism:
    $
        \sM \to \uH^\ast(Y).
    $
    Hence, $Y = X^{p^k}$ is equivariantly formal.
    
    The converse follows directly from \cref{rem: fixed point set of eq formal is formal}.
\end{proof}

\begin{example}
    Let $G = \C_{p^k}$ act on $X^{p^k}$ by permuting the factors, where $X = K(\ZZ,2)$. We describe the equivariant minimal model $\sM$ as follows.

    At the orbit level $G/p^i$, we have $\sM(G/p^i) =\QQ (\alpha)^{\otimes p^i}$, where $|\alpha|=2$. The group action and structure maps are defined by:
       $$ \sM(\hat \xi)(\alpha_1 \otimes \cdots \otimes \alpha_{p^i}) = \alpha_{p^i}\otimes \alpha_1 \otimes \alpha_2 \otimes \cdots \otimes \alpha_{p^i-1},$$
    
        $$\sM(\hat e_{p^i,p^{i-1}})(\alpha_1 \otimes \cdots \otimes \alpha_{p^i}) = 
        \bigotimes_{r=1}^{p^{i-1}} \left( \alpha_r \alpha_{p^{i-1} + r} \alpha_{2p^{i-1} + r} \cdots \alpha_{p^i - p^{i-1} + r} \right).$$   
    Since $X=K(\ZZ,2)$ is a formal space, the space $X^{p^k}$ with the prescribed $G$-action is equivariantly formal by \cref{permutation formal}.
\end{example}
\subsection{Wedges.}
Similar to the diagonal action on a product, one can define a diagonal action on wedge $\bigvee_{i=1}^m X_i$ by identifying it with the subspace $\bigcup_{i=1}^m (\ast\times\cdots \times X_i\times \cdots\times \ast) \subseteq \prod_{i=1}^m X_i$ for the based spaces $X$ and $Y$.

\begin{proposition}
  Let $G$ be a finite group acting on based spaces $X_i$ for $i=1,\dots, m$. If each $X_i$ is equivariantly formal, then the wedge $\bigvee_{i=1}^m X_i$ equipped with the diagonal action is also equivariantly formal.  
\end{proposition}
\begin{proof}
It is enough to prove the proposition for two spaces, say $X$ and $Y$. 
    Let the minimal models of $X$ and $Y$ be given by the quasi-isomorphisms $\rho: \sM \to \u{\s E}(X)$ and $\rho': \sN \to \u{\s E}(Y)$. These induce a quasi-isomorphism $\rho\oplus \rho': \sM\oplus \s N \to \u{\s E}(X)\oplus \u{\s E}(Y)=\u{\s E}(X\vee Y).$ Here, the direct sums are taken over the base field $\QQ.$

    Since $X$ and $Y$ are equivariantly formal, there exist quasi-isomorphisms $\sigma:\sM \to \s I(\uH^\ast(X))$ and $\sigma':\sN \to \s I(\uH^\ast(Y))$.
    These induce a quasi-isomorphism $\sigma\oplus \sigma': \sM\oplus \s N \to \s I(\uH^\ast(X))\oplus \s I(\uH^\ast(Y)).$

    Putting these together, we obtain a zig-zag of quasi-isomorphisms   
    $$\u{\s E}(X\vee Y)\xleftarrow[]{\rho\oplus\rho'} \sM\oplus \s N \xra{\sigma\oplus \sigma'} \s I(\uH^\ast(X))\oplus \s I(\uH^\ast(Y))\xleftarrow[]{} \uH^\ast(X)\oplus \uH^\ast(Y)=\uH^\ast(X\vee Y).$$
    
    Here, the third map is obtained by taking the direct sum of the inclusions into the injective envelopes.
    By \cref{prop: formal defn}, it follows that $X\vee Y$ is equivariantly formal.
\end{proof}

Let $G$ be a finite group and let
$Y=\bigvee_{g\in G} X_g$
be the wedge of $|G|$ copies of a simply connected based space $X$, indexed by the elements of $G$. We can define a $G$-action on $Y$ by an action on the coordinates
$h\cdot X_g=X_{hg},$ for $ h,g\in G.$

\begin{theorem}\label{thm: formality of G wedge}
Let $G$ be a finite group and let
$Y=\bigvee_{g\in G} X_g$
be the wedge of $|G|$ copies of a simply connected based space $X$, indexed by the elements of $G$, equipped with the above $G$-action.
 If $X$ is non-equivariantly formal, then $Y$ is equivariantly formal as a $G$-space.
\end{theorem}
\begin{proof}
Let $H\leq G$ be a nontrivial subgroup. We claim that $Y^H= \ast.$
Indeed, let $y\in Y^H$. If $y\neq \ast$, then $y$ belongs to a unique wedge summand $X_g$ for some $g\in G$. Since $y$ is fixed by every element of $H$, we have $h\cdot y=y$
for all $h\in H$. On the other hand, by definition of the action, $h\cdot y\in X_{hg}.$
Since distinct wedge summands intersect only at the wedge point and $y\neq \ast$, it follows that $X_{hg}=X_g$ for every $h\in H$. Hence $hg=g$ for every $h\in H$.

However, the action of $G$ on itself by left translation is free. Therefore 
$hg=g \implies h=e.$
It follows that every element of $H$ is the identity, contradicting the assumption that $H$ is nontrivial. Thus no non-basepoint of $Y$ can be fixed by $H(\neq e)$, and therefore $Y^H= \ast.$
On the other hand, $Y^e=Y.$

Since $X$ is non-equivariantly formal, the finite wedge
$Y=\bigvee_{g\in G}X_g$
is also non-equivariantly formal. Therefore, there exists a quasi-isomorphism
$\sE(Y)\xleftarrow{\rho} M\xrightarrow[]{\rho'} \H^\ast(Y).$
By \cref{thm: minimal model for each level}, one may take the quasi-isomorphisms $\rho$ and $\rho'$
 as $G$-maps.
For every nontrivial subgroup $H\leq G$, we have
$Y^H=\ast,$
and consequently
$\sE(Y^H)=\QQ=\H^\ast(Y^H).$
Using \cref{def: inj enlargement}, we obtain the injective envelope $\s I(\uH^\ast(Y))$  of $\uH^\ast(Y)$ is $\uH^\ast(Y)$ itself.
Therefore, it is easy to assemble together (since each $G/H$-level for $H\neq \{e\}$ is just a copy of $\QQ$) as in \cref{prop: lower fixed DGA minimal} and get quasi-isomorphisms of $\mathscr O_G$-$\DGA$s
$\u\sE(Y)\xleftarrow{\rho} \s M\xrightarrow[]{\rho'} \s I(\uH^\ast(Y)).$
where $\s M$ is a minimal system.
Therefore, we conclude that $Y$ is equivariantly formal.
\end{proof}

\subsection{Retracts.}
It is well known that a retract of a formal space is formal non-equivariantly. We show this result extends in the equivariant setting.
\begin{theorem}\label{thm: retraction of eq formal}
    Let $G$ be a finite group. Let $X$ be an equivariant formal $G$-space;
     then the equivariant homotopy retract of $X$ is equivariantly formal.
\end{theorem}
\begin{proof}
    Let $r: X\ra Y$ be a homotopy retraction. Then there exists $i: Y\ra X$ such that $ri\simeq id_Y$ (equivariantly).  
    Applying the contravariant functor $\underline{\mathscr E}$, we have $\underline{\s E}(r\circ i) \simeq \id$, which means  $\underline{\s E}(i)\circ \underline{\s E}(r) \simeq \id$, Since $\underline{\sE}$ preserves homotopy (see \cite[Corollary 4.11]{TR1982}).
    
    Let $m_Z:\s M_Z \ra \underline{\sE}(Z)$ denotes the minimal model of $Z$, where $Z=X,Y$. 
    Using the lifting property given in \cref{prop: lifting property}, we obtain the following homotopy commutative diagram:

\begin{equation}\label{eq: formality of retract}
\begin{tikzcd}
	{\s M_Y} & {\s M_X} & {\s M_Y} \\
	{\underline {\mathscr E}(Y)} & {\underline {\mathscr E}(X)} & {\underline {\mathscr E}(Y)}
	\arrow["{\tilde r}", dashed, from=1-1, to=1-2]
	\arrow["{m_Y}", from=1-1, to=2-1]
	\arrow["{\tilde i}", dashed, from=1-2, to=1-3]
	\arrow["{m_X}", from=1-2, to=2-2]
	\arrow["{m_Y}", from=1-3, to=2-3]
	\arrow["{\underline {\mathscr E}(r)}", from=2-1, to=2-2]
	\arrow["{\underline {\mathscr E}(i)}", from=2-2, to=2-3]
\end{tikzcd}
\end{equation}
    where $\tilde r$ and $\tilde i$ lift $r$ and $ i,$ respectively.
    We claim that $\tilde i \circ \tilde r\simeq \id$. 
    We note that there exists a homotopy $H: \s M_Y \ra \underline {\mathscr E}(Y)\otimes \u \QQ(t,dt) $ with $$p_0 H =\underline{\s E}(i) \underline{\s E}(r) m_Y ,\qquad p_1H= \id.$$ 
    Using \cref{prop: lifting property}, we lift the homotopy with respect to $m_Y\otimes \id: \s M_Y\otimes \u \QQ(t,dt) \ra \underline{\sE}(Y)\otimes \u \QQ(t,dt) $ to obtain $\tilde H:\s M_Y\ra \s M_Y \otimes \u \QQ(t,dt).$ Since $\tilde i \tilde r $ and $p_0\tilde H$ both lift $\underline{\s E}(i) \underline{\s E}(r) m_Y$ with respect to $m_Y$, by the uniqueness of lifts (up to homotopy), we have $\tilde i \tilde r \simeq p_0\tilde H$. By the same reasoning, $\id \simeq p_1\tilde H$. Combining these two homotopies, we conclude that $\tilde i \circ \tilde r\simeq \id$.
    
    On cohomology, this becomes $\tilde i^\ast \tilde r^\ast=\id$ since homotopic maps induce equality on cohomology. To justify this, suppose  $f_0, f_1$ are maps that are homotopic via $K:\sA \ra \s B\otimes \u \QQ(t,dt)$; then $p_0K=f_0, p_1K=f_1$. Since  $\uH^\ast (\u\QQ(t,dt))=\u\QQ$, the  K\"unneth formula implies $p_0^\ast=\id= p_1^\ast.$ Therefore $$f_0^\ast= p_0^\ast K^\ast=p_1^\ast K^\ast=f_1^\ast.$$
    
   Since $X$ is an equivariantly formal space, there exists a quasi-isomorphism given below: $$\sigma: \sM_X \ra \s I(\uH^\ast(\s M_X))\cong\s I(\uH^\ast (X)).$$
   We note that $\uH^\ast(\s I(\uH^\ast (\s A))=\uH^\ast(\s A)$ for any $\s O_G$-$\DGA$ by the construction of injective enlargement as in \cref{def: inj enlargement}.
   Therefore, $\sigma$ induces $\sigma^\ast :\uH^\ast(\s M_X)\ra \uH^\ast(\s M_X)$. Taking injective envelope, we obtain a map $\s I((\sigma^\ast)^{-1}): \s I(\uH^\ast(\s M_X))\ra \s I(\uH^\ast(\s M_X))$ which extends $(\sigma^\ast)^{-1}.$ 
   Let us take $$\rho=  \s I((\sigma^\ast)^{-1})\circ \sigma:\sM_X \ra \s I(\uH^\ast(\s M_X))$$ then $\rho^\ast=\id:\uH^\ast(\s M_X)\ra \uH^\ast(\s M_X) $.
   Let $\psi$ be the composite morphism given by the maps
   $$\psi:\s M_Y \xra{\tilde r} \s M_X\xra{\rho} \s I(\uH^\ast(\s M_X)) \xra{\s I(\tilde i^\ast)}\s I(\uH^\ast(\s M_Y)).$$
   where $\s I(\tilde i^\ast)$ denotes the extension of $\tilde i$ to the injective envelopes.
   On cohomology, this becomes $$\psi^\ast= \tilde i^\ast \rho^\ast \tilde r^\ast= \tilde i^\ast \tilde r^\ast=\id.$$ Thus $\psi$ is a quasi-isomorphism, hence $Y$ is equivariantly formal.
    
\end{proof}

As a corollary, we obtain the converse to \cref{prop:product of equivariantly formal spaces is formal} and \cref{thm: formality of G wedge}.
\begin{corollary}
      Let $G$ be a finite group. Suppose $G$ acts on $X\times Y$ by the diagonal action induced from the action on individual components. If $X\times Y$ is formal, then the components $X$ and $Y$ are equivariantly formal.
\end{corollary}
\begin{proof}
    We have two $G$-maps. The maps $\pi_1:X \times Y\to X$ and $\pi_2:X \times Y\to Y$ are the projection on the components, and $i_1:X_1\ra X\times Y$ and $i_2:Y\ra X\times Y$ are the inclusions. The maps satisfy the relation $\pi_ji_j=\id$; thus, $X, Y$ are equivariant retracts of $X\times Y$. Therefore, by \cref{thm: retraction of eq formal}, $X$ and $Y$ are equivariantly formal. 
\end{proof}

\begin{corollary}
    Let $G$ be a finite group and let
$Y=\bigvee_{g\in G} X_g$
be the wedge equipped with a $G$-action on $Y$ by 
$h\cdot X_g=X_{hg},$
 for $h,g\in G.$ If $Y$ is equivariantly formal as a $G$-space, then $X$ is non-equivariantly formal.
\end{corollary}
\begin{proof}
    By \cref{cor: restriction preserves formality}, $Y$ is non-equivariantly formal. Now, the following composition: $$X \to \bigvee_{g\in G}  X_g\xra{\mathrm{fold}}X$$ is the identity; that is, the folding map is a retraction. Therefore, using \cref{thm: retraction of eq formal}, X is formal.  
\end{proof}

\subsection{Connected sums.}
The connected sum of simply connected compact manifolds is formal (see \cite[Theorem 3.13]{FOT}). We now see an involution on a connected sum which does not yield an equivariantly formal space. The example below contrasts with \cref{thm: formality of G wedge}, where the involution obtained by permuting wedge summands gives an equivariantly formal space. 

\begin{example}\label{ex: join is not eqformal}
Let $X=\CC P^2 \# \CC P^2$ be a $\C_2$-space, where the generator $\xi$ acts on the space by swapping the two copies. First, we note that $X^{\C_2}=S^3$ and the equivariant cohomology is given by $$\uH^\ast (X)= \begin{cases}
\QQ(x, y)/(xy,x^2-y^2) , &\text{at}\ \C_2/e\\
\QQ(z), &\text{at}\ \C_2/\C_2
\end{cases}$$
where $|x|=|y|=2, |z|=3.$ The structure map induces $\uH^\ast(X) (\hat e_{e,\C_2})=0$ for degree reasons, and the action of the generator at the $\C_2/e$ level is given by $\xi x=y$.

Since $X$ is non-equivariantly formal, its non-equivariant minimal model can be obtained by taking the minimal model of $\H^\ast(X).$
The minimal model of $X$ equipped with a $\C_2$-action, as constructed in \cref{thm: minimal model for each level}, is given by $M=\QQ(x,y,u,v)$ with $ du=x^2-y^2, dv=xy$, the actions are $$\xi x=y,\quad \xi y=x,\quad \xi u=-u,\quad \xi v=v.$$ We note that $N:= \QQ(z), dz=0$ is the minimal model for $S^3.$

Let us write $X=X_1\#X_2.$
 We note that the projective space is obtained by attaching $4$-cells to $S^2$ via the Hopf fibration $h: S^3\to S^2$; that is, $X_i=\CC P^2=S^2\cup_h e^4$.
The inclusion $X^{\C_2}\ra X$ factors through $X_1\backslash {\mathring e}^4$, as given by the diagram
$$\begin{tikzcd}
	{S^3=X^{\C_2}} && X \\
	& {S^2\simeq X_1\backslash {\mathring e}^4}
	\arrow["i", from=1-1, to=1-3]
	\arrow["h"', from=1-1, to=2-2]
	\arrow["j"', from=2-2, to=1-3]
\end{tikzcd}$$
Let the minimal model of $S^2$ be $P=\QQ(a,b)$ with $|a|=2$ and $db=a^2.$ Then $h$ and $j$ induce maps on the minimal models, $ \tilde h: P\ra N$ and $\tilde j:M\ra P$, given by
$$\tilde h(a)=0,\quad \tilde h(b)=z.
$$
$$\tilde j(x)=a ,\quad \tilde j(y)=0 ,\quad \tilde j(u)=b ,\quad \tilde j(v)=0 $$
To justify this, we note that the Hopf map $h:S^3\to S^2$ induces an isomorphism on the third homotopy groups. Moreover, the vector space of generators of a minimal model of a space $X$ in degree $n$ is naturally isomorphic to the dual space of the rational homotopy group $\hom(\pi_n(X) \otimes \mathbb{Q}, \mathbb{Q})$, (see \cite[Theorem 15.11]{FHT}). In particular, this ensures $\tilde h$ maps the degree $3$-generator $b$ to $z$. For $\tilde j$, a similar argument applies once we note that the $2$-skeleton of $X$ is $S^2 \vee S^2$, and the map $j$ can be identified with the inclusion of $S^2$ as the first summand of the wedge.

Hence, the composite $\tilde i: M\ra N$ is given by $$\tilde i(x)=\tilde i(y)=\tilde i(v)=0,\quad \tilde i(u)=z.$$ Here, we note that the topological relation $j\circ h\simeq i$ implies that the induced maps on the minimal models are homotopic, namely $\tilde h\circ\tilde j\simeq \tilde i$. However, algebraic rigidity forces $\tilde h \circ \tilde j$ to be the unique map in its homotopy class. This justifies the strict lift $\tilde i$ of $i$ on the minimal models as the composition $\tilde h\circ\tilde j$.

Let $\s B$ be an $\s O_G$-$\DGA$ defined by $\s B(\C_2/e)=M$ and $\s B(\C_2/\C_2)= N$, with the structure map $\s B(\hat e_{e,\C_2})=\tilde i.$
Putting together, the minimal model of the layers $\s E(X^H)$ for $H=e, \C_2$ produce a quasi-isomorphism $\s B\xra{\sim} \u{\s E}(X).$

The $\s O_G$-$\DGA$ $\s B$ is not injective in $\vc^\ast;$ we modify this by introducing extra generators at the $\C_2/e$-level to obtain the equivariant minimal model $\s M$ of $X$,
$$\s M=\begin{cases}
\QQ(x,y,u,v,z,z'), &\text{at}\ \C_2/e\\
\QQ(z), &\text{at}\ \C_2/\C_2
\end{cases}$$
with the differentials at $\C_2/e$ given by $ du=x^2-y^2, dv=xy, dz=z'$, and at $\C_2/\C_2$ by $dz=0.$ The degrees and actions are defined as before.
We claim that the space $X$ is not equivariantly formal; that is, there is no quasi-isomorphism $\s M\ra \s I(\uH^\ast (X)).$
By \cref{def: inj enlargement}, the injective envelope is given by $$\s I(\uH^\ast (X))=\begin{cases}
\QQ(x, y)/(xy,x^2-y^2)\otimes \QQ(z,sz) , &\text{at}\ \C_2/e\\
\QQ(z), &\text{at}\ \C_2/\C_2
\end{cases}$$
where $|z|=3$, and $dz=sz.$
For the sake of contradiction, assume there exists a quasi-isomorphism $\rho: \s M\ra \s I(\uH^\ast (X)).$
At the $\C_2/e$-level, the quasi-isomorphism forces $\rho(x)=ax+by$ and $\rho(y)=cx+dy$ for some basis $\{(a,b),(c,d)\}$ of $\QQ^2$. Since $du=x^2-y^2$ in $\sM(\C_2/e)$, this implies $\rho(\C_2/e)(du)=\rho(\C_2/e)(x^2-y^2)$ a polynomial in $x$ and $y$. On the other hand, degree considerations imply that $\rho(\C_2/e)(u)=kz$ for some $k\in \QQ,$ so $\rho(\C_2/e)(du)=ksz$. So the only solution to this is $k=0$, meaning, $\rho(\C_2/e)(u)=0.$
However, $$\rho(\C_2/\C_2)\circ \s M(\hat e_{e, \C_2})(u)=\rho(\C_2/\C_2)(z)=z.$$ On the other hand $\s I(\uH^\ast(X))(\hat e_{e,\C_2})\circ \rho(\C_2/e)(u)=\s I(\uH^\ast(X))(\hat e_{e,\C_2})(0)=0.$ This implies $$\rho(\C_2/\C_2)\circ \s M(\hat e_{e, \C_2})\neq\s I(\uH^\ast(X))(\hat e_{e,\C_2})\circ \rho(\C_2/e)$$
which contradicts the naturality of $\rho.$ Thus,  no such quasi-isomorphism exists. Hence $X$ is not equivariantly formal.
\end{example}

\iffalse
\begin{remark}
Halperin and Stasheff's result \cite[Corollary 5.16]{HalpStasheff1979} shows that if a cohomology algebra $A$ satisfies $A^i=0$ for $1\leq i\leq d$ and for $i\geq 3d+1$, then $A$ is intrinsically formal.
In \cref{ex: join is not eqformal}, the cohomology algebra satisfies the algebraic restriction at each $G/H$-level, but the corresponding space is not equivariantly formal.   Therefore, the cohomology algebra is not equivariantly intrinsically formal.
\end{remark}
\fi
This example also illustrates that some of the significant theorems on formality fail to have obvious generalizations in the equivariant setting.  
\begin{remark}
In the non-equivariant world, there are dimensional connectivity results that ensure the formality of spaces.
It is known that $n$-connected CW-complex of  dimension $\leq 3n +1$ is formal for $n\geq 1$ (see \cite[Corollary 2.4]{ShigaYamguchi},\cite[Corollary 5.16]{HalpStasheff1979}). 
In \cref{ex: join is not eqformal}, the fixed point spaces $X^H$ satisfy these condition as
$\dim(X)=4\leq3\cdot 1+1$ and $\dim(X^{\C_2})=3\leq 3\cdot 2+1$.
However, the space $X$ is not equivariantly formal.
\end{remark}
\begin{remark}
 Miller \cite{miller1979formality} showed that an $n$-connected compact manifold with dimension $\leq 4n +2$ is formal for $n\geq 1$. Again, \cref{ex: join is not eqformal} shows that this theorem does not generalize to the equivariant setting.
\end{remark}

\section{Intrinsic formality in the case of \texorpdfstring{$\C_{p^k}$}{primary cyclic group}}\label{sec: Intrinsic formality}
 Recall that a graded algebra $A^{\ast}$, is said to be $k$-intrinsically formal, if for any minimal algebra $M$ with $H^{\ast}(M)\cong A^{\ast}$ the sub $\DGA$ $M_k$ of $M$ generated by elements degree $\leq k$, is unique up to isomorphism. If $M_k$ is unique for every $k$, the graded algebra $A^{\ast}$ becomes intrinsically formal, and any space with cohomology algebra $A^{\ast}$ will be formal. The equivariant analogue was defined in \cite{santhanam2025equivariant}.

\begin{definition}[Equivariantly intrinsically formal]\label{def:I.F}
A diagram of graded algebras $\sA$ over $\sO_G$, for some finite group $G$ is called \emph{equivariantly $k$-intrinsically formal} (abbreviated to equivariantly $k$-I.F.) if for any two minimal systems $\sM$ and $\sM'$ with $H^\ast(\sM)=\sA=H^{\ast}(\sM')$, the sub-systems $\sM_k$ and $\sM'_k$ have the property that, there is a morphism from $\sM_k \to \sM'_k$ or $\sM_k'\to \sM_k$ which is a $k$-isomorphism.
 For any two such minimal systems $\sM$ and $\sM'$ if there is a morphism $\phi:\sM \to \sM'$ which is a quasi-isomorphism, we say $\sA$ is \emph{equivariantly intrinsically formal} (abbreviated to equivariantly I.F.).
 \end{definition}
Note that if $\uH^\ast(X)$ is equivariantly I.F., there exists a unique minimal system $\s M$ whose cohomology agrees with $\uH ^\ast(X)$, then it is forced that $\s M$ is the minimal model of $\uH ^\ast(X)$.
Hence if $\uH^\ast(X)$ is equivariantly I.F., $X$ is equivariantly formal.

We now describe a criterion for deciding whether a split-$\s O_{\C_{p^k}}$-$\GA$ admits multiple equivariant rational homotopy types. That is, we determine whether there exist two $\C_{p^k}$-spaces with the same cohomology diagram but different equivariant rational homotopy types. Here, we modify the techniques developed in \cite{ShigaYamguchi,santhanam2025equivariant} to get a better understanding of plural homotopy types.

\begin{cons}\label{cons: for plural homotopy types}
    Let $\s A$ be a split-$\s O_{G}$-$\GA$ over $\s O_{G}$ where $G={\C_{p^k}}$.  
    We denote  $\s A(n)$ to be the subalgebra of $\s A $ multiplicatively generated by elements up to degree $n$.

We declare a minimal systems \textbf{$\s M_n$ 
satisfies the properties $\mathbf{(1)_n, (2)_n}$} if the following conditions holds:\\
   $\mathbf{(1)_{n}}$ each ${\C_{p^k}}/{\C_{p^i}}$-orbits $\sM_n({\C_{p^k}}/{\C_{p^i}})$ are generated by elements of degree $\leq n.$\\
    $\mathbf{(2)_{n}}$ There is a morphism of diagram of GAs $$\sigma_{n}: \uH^\ast(\sM_{n})(n+1)\ra \s A$$ where $\sigma_{n}^i$ is isomorphic for degree $i\leq n$ and monomorphic for $i=n+1$ and $\ker(\sigma_{n})^{n+2}$ as an element of $\vc^\ast$ is injective.

We note that, since $\s A$ is a split-$\s O_{G}$-$\GA$, it is system of $\DGA$s with zero diffrential. Consequently, the cohomology of $\s A$ is naturally isomorphic to $\s A$ itself.
Note  that the $n$-th stage of the minimal model of the split-$\s O_{\C_{p^k}}$-$\GA$ $\s A$ satisfies $\mathbf{(1)_{n}}, \mathbf{(2)_{n}}$ and therefore such a system always exists.

Let  $\s M_{n-1}$ be a minimal system satisfying $\mathbf{(1)_{n-1}}, \mathbf{(2)_{n-1}.}$ We now construct an elementary extension  of $\sM_{n-1}$ denoted by the miminal system  $\s M_n$ containing $\s M_{n-1}$ and satisfying the conditions $\mathbf{(1)_{n}}, \mathbf{(2)_{n}.}$

\textbf{Step 1.} 
We denote $\uV^n:=\ker(\sigma_{n-1})^{n+1}$ which is injective as an element of $\vc ^\ast$ by hypothesis. 
Let $\s M_D$ be the elementary extension $\s M_{n-1}(\u V^n)$.
We have the induced morphism $\sigma_D :\uH^\ast(\s M_{D})(n) \ra \s A.$ 

\textbf{Step 2.}
Let $\alpha, \beta, \gamma, \delta$ denote the $(k+1)$-tuples (for example $\alpha=(\alpha_0, \dots \alpha_k))$ obtained from the dimensions of the following diagrams of $\GA$s evaluated at each orbit:
\begin{align*}
    &\alpha_i=\dim_{\QQ} \sA^{n+1}({\C_{p^k}}/{\C_{p^i}}),  \hspace{3.45em} \gamma_i=\dim_{\QQ} (\sA^{n+1}({\C_{p^k}}/{\C_{p^i}})/\sA(n)^{n+1})({\C_{p^k}}/{\C_{p^i}})
    ,\\
    &\beta_i=\dim_{\QQ} \uH^{n+1}\s (\sM_D)({\C_{p^k}}/{\C_{p^i}})
   ,\quad \delta_i=\dim_{\QQ} (\uH^{n+1}\s (\s M_D)({\C_{p^k}}/{\C_{p^i}})/(\uH^\ast(\s M_{D})(n))^{n+1}({\C_{p^k}}/{\C_{p^i}}).
\end{align*}
Clearly, $\sA^{n+1}/\sA(n)^{n+1}$ is the subspace of $n+1$ degree elements of $\sA/\sA(n)$.
Here, the quotients are the cokernels of the respective inclusions, i.e., $\sA(n)\to \sA$ and $\uH^\ast(\s M_{D})(n)\to \uH^\ast(\sM_D)$.
Note that we have the level-wise difference $\alpha_i -\gamma_i=\beta_i-\delta_i$ 
since the respective decomposable parts up to degree $n+1$ (such decomposable elements are a sum of products of elements up to degree $(n-1)$) are isomorphic, as $\sigma_{n-1}^i$ is isomorphic for degree $i\leq n-1$. 
Let $\kappa$ be a $(k+1)$-tuple such that $$\max\{0, \delta-\gamma\}\leq \kappa \leq \delta.$$
Let there  exist $\u W\subseteq \uH^\ast(\s M_D)$, an injective $\vc^\ast$  with $\dim_{\QQ} \u W=\kappa$ such that 
\begin{equation}\label{eq: transversility condition}
    \u W \cap (\uH^\ast(\s M_{D})(n))^{n+1}=0 ~~\text{(level wise)}. 
\end{equation}

Let $\s M_{\u W}$ denotes the minimal system $\s M_D(\u W')$, where $\u W'$ is an isomorphic copy of $\u W$ is added in degree $n$, by  setting $d(\u W')=\u W.$
Since $\uH^\ast(\s M_{\u W})  (n)=\uH^\ast(\s M_{D})(n)$, we have the induced morphism $\sigma_D:\uH^\ast(\s M_{\u W})(n) \ra \s A$. Also we  have $$\uH^{n+1}(\s M_{\u W})\oplus\u W=\uH^{n+1}(\s M_D)$$  so that
$$\dim_{\QQ} (\uH^{n+1}\s (\s M_{\u W})/(\uH^\ast(\s M_{\u W})(n))^{n+1}=\delta_i-\kappa_i\leq \gamma_i=\dim_{\QQ} (\sA^{n+1}/\sA(n))^{n+1} $$ at each orbit level $(C_{p^k}/C_{p^i})$. This is  because \eqref{eq: transversility condition} implies only non-decomposable elements of $\uH^{n+1}\s (\s M_D)/(\uH^\ast(\s M_{D})(n))^{n+1}$ are killed by adding $\u W'$.

\textbf{Step 3.}
Let $\u C=\coker((\sigma_D)^n:(\uH^\ast(\s M_{\u W})(n))^n \ra \s A^n)$. Here, we note that by \cref{cpk resolution length 1} $\coker((\sigma_D)^n$ is injective in $\vc^\ast$ as $\s A^n$ is injective in $\vc^\ast$.
Let $\s M_{\u W, n}=\s M_{\u W}(\u C)$ with $d(\u C)=0.$ Then, we have a $\DGA$
morphism $\sigma_n:\uH^\ast(\s M_{\u W, n})(n)\ra \s A$ which is isomorphic for degree $i\leq n.$ 

\textbf{Step 4.}
We say that the system of rational vector spaces $\underline{W}$ satisfies $\mathbf{(C_{n-1})}$ if the following holds.\\
\textbf{Condition $\mathbf{(C_{n-1})}: $} Given  a monomorphism of rational coefficient systems $$\psi:\uH^{n+1}(\sM_{\u W})/\uH^\ast(\s M_{\u W})(n))^{n+1}\ra \s A^{n+1}/\s A(n)^{n+1},$$
 the morphism $\sigma_n\oplus\psi$ can be extended to a morphism of a diagram of GAs
$\sigma_{\u W,n}:\uH^\ast(\s M_{\u W, n})(n+1)\ra \s A$
with $\ker \sigma_{\u W,n}^{n+2}$ injective in $\vc^\ast$.

We note that if $\underline{W}$ satisfies  $\mathbf{(C_{n-1})} $ then the pair $(\s M_{\u W, n}, \sigma_{\u W,n})$ satisfies the condition $\mathbf{(1)_n}$ and $\mathbf{(2)_n}$.

\textbf{Step 5.}
However, two $\s M_{\u W, n}$ obtained through this construction may not represent different equivariant rational homotopy types. Now we deal with this problem.
Let us define $$\mathbb G(\beta,\kappa)(\QQ)=\{\u W \mid  \u W \subseteq H^\ast (\s M_D)(n)^{n+1} \text{ with }  \dim_{\QQ} \u W= \kappa\},$$ 
and  $$\mathbb M_\kappa:=\{\u W\in\mathbb G(\beta,\kappa)(\QQ) \mid \u W \text{ is injective that satisfies } \u W \cap (\uH^\ast(\s M_{D})(n))^{n+1}=0 \}.$$
Let $\mathbb O_\kappa \subseteq \mathbb M_\kappa$ be defined as
$$\mathbb O_\kappa:=\{\u W\in \mathbb M_\kappa \mid \u W \text{ satisfies condition } \mathbf{(C_{n-1})}   \}.$$
Let $\mathscr G=\Aut (\cM_D)$ be the group of automorphism  
then $\mathscr G$ acts on $\mathbb G(\beta,\kappa)(\QQ)$ canonically. We will show that two choices of $\s M_{\u W, n}$ obtained through the construction are equivalent if and only if they lie in the same $\mathscr G$-orbit.
\begin{itemize}
    \item 
Let $\Phi $ be an element of $\mathscr G$.
Since an automorphism preserves injectivity, and decomposable and indecomposable elements are mapped to decomposable and indecomposable elements respectively, so the condition \eqref{eq: transversility condition} is satisfied by $\Phi^\ast(\u W)$ (here $\Phi^\ast$ is the induced morphism in cohomology) if it is satisfied by $\u W$. Therefore, the action of $\mathscr G$ restricts to $\mathbb M_\kappa$. Also, if $\u W$ satisfies $\mathbf{(C_{n-1})}$ that is there exists an extension of the morphism of diagram of GAs $\sigma_n\oplus \psi$, then we have an extension of morphism of diagram of GAs $\Phi^\ast(\sigma_n) \oplus \Phi^\ast(\psi)$, so $\Phi^\ast(\u W)$ satisfies $\mathbf{(C_{n-1})}$, implying, the action of $\mathscr G$ restricts to $\mathbb O_\kappa$.

Moreover, for $\u W \in \mathbb O_\kappa$
the automorphism $\Phi$ extends uniquely to an $\s O_G$-$\DGA$ isomorphism $\hat \Phi: \c M_{\u W,n}\ra \c M_{\Phi^\ast(\u W), n}.$ Therefore, minimal systems corresponding to points in the same $\mathscr G$ orbits in $\mathbb O_\kappa$ are unique up to isomorphism.

\item Conversely, assume $\c M_{\u W_1,n}$ and $\c M_{\u W_2, n}$ are isomorphic minimal systems given by an isomorphism $\Psi :\c M_{\u W_1,n} \ra \c M_{\u W_2, n}$. As the construction of both of $\c M_{\u W_1,n}$ and $\c M_{\u W_2, n}$ must factor through $\c M_D$, so $\Psi|_{\c M_D}$ is an element of $\mathscr G$.

We also note that $\u W_i=\ker(j^\ast:\uH^{n+1}(\c M_D) \ra \uH^{n+1}(\c M_{\u W_i,n}))$ for $i=1,2$. Therefore, the following commutative diagram:
$$\begin{tikzcd} 
	0 & {\u W_1} & {\uH^{n+1}(\c M_D)} & {\uH^{n+1}(\c M_{\u W_1,n})} \\
	0 & {\u W_2} & {\uH^{n+1}(\c M_D)} & {\uH^{n+1}(\c M_{\u W_2,n})}
	\arrow[from=1-1, to=1-2]
	\arrow[from=1-2, to=1-3]
	\arrow[from=1-2, to=2-2]
	\arrow["{j^\ast}", from=1-3, to=1-4]
	\arrow[from=1-3, to=2-3]
	\arrow["{{\Psi^\ast}}"', from=1-4, to=2-4]
	\arrow[from=2-1, to=2-2]
	\arrow[from=2-2, to=2-3]
	\arrow["{j^\ast}", from=2-3, to=2-4]
\end{tikzcd}$$
of exact sequences
where the first two vertical maps are restriction maps of $\Psi^\ast$ and the first horizontal maps are inclusions of the kernel.  As the last two vertical maps are isomorphisms, so $\Psi^\ast:\u W_1 \ra \u W_2$ is also an isomorphism by diagram chasing.
\end{itemize}
\end{cons}

Therefore, we have the following.
\begin{theorem}\label{thm: plural homotopy types in C}
Let $\s A$ be a  split-$\s O_{\C_{p^k}}$-$\GA$, where $p$ is a prime number.
    The set of isomorphism classes of equivariant minimal systems $\sM_n$  containing $\sM_{n-1}$ and satisfying $\mathbf{(1)_n, (2)_n}$ obtained through the \cref{cons: for plural homotopy types} corresponds bijectively the disjoint union of orbit spaces $$\mathbb X_n=\bigsqcup_{\kappa=\max\{0, \delta-\gamma\}}^\delta \mathbb O_\kappa/\s G.$$
\end{theorem}

\begin{example}
Consider the split-$\s O_{\C_{p^2}}$-$\GA$s
$$\sA=
\begin{cases}
\QQ(x_3,y_7,z_{22})/(xy,xz,yz,z^2), & \text{at } \C_{p^2}/e,\\
\QQ(x_3), & \text{at } \C_{p^2}/\C_p,\\
\QQ(x_3), & \text{at } \C_{p^2}/\C_{p^2},
\end{cases}$$
where the degrees of the generators are indicated by the subscripts. The structure maps are given by
$$\sA(\hat e_{e,\C_p})(x_3)=x_3,\qquad
\sA(\hat e_{e,\C_p})(y_7)=
\sA(\hat e_{e,\C_p})(z_{22})=0,$$
and
$\sA(\hat e_{\C_p,\C_{p^2}})(x_3)=x_3.$

The $20$th stage of the minimal model is
$$\sM_{20}(G/H)=
\begin{cases}
\QQ(x,y,\theta_9,\theta_{11},\theta_{13},\theta_{15},
\theta_{15}',\theta_{17},\theta_{17}',
\theta_{19},\theta_{19}'),
& \text{at } \C_{p^2}/e,\\
\QQ(x_3), & \text{at } \C_{p^2}/\C_p,\\
\QQ(x_3), & \text{at } \C_{p^2}/\C_{p^2},
\end{cases}$$
with differential
\begin{align*}
    &dx=dy=0, d\theta_9=xy,d\theta_{11}=x\theta_9,d\theta_{13}=x\theta_{11},d\theta_{15}=y\theta_9,d\theta_{15}'=x\theta_{13},\\
    &d\theta_{17}=x\theta_{15}+y\theta_{11},d\theta_{17}'=x\theta_{15}',d\theta_{19}= x\theta_{17}+y\theta_{13},d\theta_{19}'=x\theta_{17}'.
\end{align*}

Since $\sM_{20}(G/H)$ is already the non-equivariant minimal model of
$\sA(G/H)$ for every $H\neq e$, it suffices to focus on  the
$\C_{p^2}/e$ level in order to examine if the next step admits plural homotopy type at $\C_{p^2}/e$ level.
As shown in \cite[Example~1]{ShigaYamguchi},
there exist at least two distinct choices for extending $\sM_{20}$ to
$\sM_{21}$. These extensions yield non-isomorphic minimal systems.
Consequently, $\sA$ admits plural homotopy types.
\end{example}
Now, let us consider the special case where a system of $\GA$s $\s A$ satisfies the condition $\s A(G/H)=0$ for all $H>e$. In this case, we note that every minimal system with the cohomology $\s A$ must be obtained by the \cref{cons: for plural homotopy types}.
Since such a minimal system $\s M$ satisfies $\s M(G/H)=\QQ$ for all $ H> e$, the coefficient systems $\u V, \u W,$ and $ \u C$ appearing in the construction are zero except at $G/e$-level. 
In particular, we can construct $\u W$ by taking a $\QQ(G)$-module $W$ at $G/e$-level and $0$ at other orbit levels.
Since the structure maps of these coefficient systems are surjective, they are automatically injective in $\vc^\ast.$ Consequently, 
\textbf{Step 5} of the construction also simplifies as follows:
 $$ G(\beta_0,\kappa_0)(\QQ):=\{ W \mid   W \subseteq H^\ast (\s M_D(G/e))(n)^{n+1} \text{ with }  \dim_{\QQ}  W= \kappa_0\},$$ 
and  $$ M_{\kappa_0}:=\{ W\in G(\beta_0,\kappa_0)(\QQ) \mid   W \cap (\uH^\ast(\s M_{D}(G/e))(n))^{n+1}=0 \}.$$

$$ O_{\kappa_0}:=\{ W\in  M_{\kappa_0} \mid \u W \text{ satisfies condition } \mathbf{(C_{n-1})}   \}.$$
In this case, the automorphism group $\tilde G:=\Aut (\cM_D(G/e))\cong\Aut (\cM_D)$ acts on $O_{\kappa_0}$. As a result, we obtain a stronger characterization for this case.
\begin{theorem}\label{thm: specific plural homotopy types in C}
Let $\s A$ be a  sytem of $\GA$s over $\s O_{\C_{p^k}}$ such that $\s A(\C_{p^k}/H)=\QQ$ for all $e<H\leq\C_{p^k}$, where $p$ is a prime number.
    The set of isomorphism classes of equivariant minimal systems $\sM_n$ containing $\sM_{n-1}$ and satisfying $\mathbf{(1)_n, (2)_n}$ corresponds bijectively the disjoint union of orbit spaces $$X_n=\bigsqcup_{\kappa_0=\max\{0, \delta_0-\gamma_0\}}^\delta O_{\kappa_0}/\tilde G.$$
\end{theorem}

A natural question is whether an $\s O_G$-graded algebra consisting of non-equivariantly intrinsically formal graded algebras at each orbit level is necessarily equivariantly intrinsically formal. We provide an example demonstrating that this is far from the truth. 
\begin{example} \label{ex: IF in each level but not eq IF}
    Let $\s A$ be a systems of GAs over $\s O_{\C_p}$ given by 
    $$\sA= \begin{cases}
    \QQ(x_2, y_3)/(x_2^2) , &\text{at}\ {\C_p}/e\\  
         \QQ(y_3'), &\text{at}\ {\C_p}/{\C_p}
\end{cases}$$
where the structure maps are $\s A(\hat e_{e, {\C_p}})(x_2)=0, \s A(\hat e_{e, {\C_p}})(y_3)=y_3'$, and the degrees of the generators are indicated by the subscripts. We note that each level is non-equivariantly intrinsically formal (see \cite[Theorem 5.2]{ShigaYagita}).

Let us construct minimal system $\s M(c)$ corresponding to each rational number $c \in \QQ$ by taking elementary extension of $$\s M_2= \begin{cases}
   \QQ(x_2) , &\text{at}\ {\C_p}/e\\  
         \QQ, &\text{at}\ {\C_p}/{\C_p}
\end{cases}$$
with respect to the maps $\alpha: \u V(c)\ra \s M_2$, where 
$$\u V(c)= \begin{cases}
    \QQ(y_3, z_3) , &\text{at}\ {\C_p}/e\\  
         \QQ(y_3'), &\text{at}\ {\C_p}/{\C_p}
\end{cases}$$
with $\u V(c)(\hat e_{e,{\C_p}})(z_3)=cy_3'$, $ \u V(c)(\hat e_{e,{\C_p}})(y_3)=y_3'$,
and $\alpha$ is given by 
$$\alpha({\C_p}/e)(y_3)=0,  \alpha({\C_p}/e)(z_3)=x_2^2, \alpha({\C_p}/{\C_p}) (y_3')=0.$$
We observe that $\u V(c)$ is injective in $\vc^\ast$ because  the structure maps are surjective, therefore $$\s M(c)=\s M_2^\alpha(\u V(c))= \begin{cases}
   \QQ(x_2,y_3,z_3) , &\text{at}\ {\C_p}/e\\  
         \QQ(y_3'), &\text{at}\ {\C_p}/{\C_p}
\end{cases}$$
where the non-trivial differential appears at ${\C_p}/e$ level and is given by $dz_3=x_2^2.$ The structure map $M(c)(\hat e_{e, {\C_p}})$ is given by $$ x_2\mapsto0, y_3\mapsto y_3', z_3\mapsto cy_3'.$$
It follows that $\uH^\ast(\s M(c))=\s A.$ Now we claim that although the cohomology of the constructed equivariant minimal systems agrees with $\s A$, the minimal systems $\s M(c_0)$ and $\s M(c_1)$ represent two different homotopy types of spaces if $c_0\neq c_1.$
We show that the structure maps $\s M(c_0)(\hat e_{e, {\C_p}})$ is not homotopic to $\s M(c_1)(\hat e_{e, {\C_p}})$.
For the sake of contradiction assume there exists a homotopy $$H:\QQ(x_2,y_3,z_3)\ra  \QQ(y_3')\otimes\QQ(t,dt)$$
such that $p_i H=\s M(c_i)(\hat e_{e, {\C_p}})$ for $i=0,1.$ Degree considerations force $H(z_3)=ky_3'$ for some $k\in \QQ$.
Now  $$p_iH(z_3)=\s M(c_i)(\hat e_{e, {\C_p}})(z_3)=c_i(y_3')$$
for $i=0,1.$ But, $p_iH(z_3)=ky_3'$. Therefore $c_0y_3'=ky_3'=c_1y_3'$, which yields $c_0=c_1$, arriving at a contradiction.
Therefore, we conclude that $\s A$ is not equivariantly intrinsically formal. Moreover, there exist infinitely many distinct rational homotopy types $\s M(c)$ having same cohomology algebra $\s A$ corresponding to each $c\in \QQ.$ 

\end{example}
\begin{remark}
    In \cref{rem: fixed point set of eq formal is formal}, we observed that fixed point sets of an equivariantly formal space are non-equivariantly formal. Here we observe that the converse is not true using \cref{ex: IF in each level but not eq IF}. Let $\s M$ denote the equivariant minimal model of the cohomology algebra $\s A$, and $\s M(c)\ncong \s M$. We conclude that $\sM(c) $ represents a space that is equivariantly non-formal. If it were formal, we would
    obtain a quasi-isomorphism $\s M(c)\xra{\sim}\uH^\ast(\s M(c))=\s A.$ 
    However, since $\s M$ is the minimal model for $\s A$, we also have a quasi-isomorphism $\s M \xra{\sim} \s A$. Therefore, $\sM(c)$ would be weakly equivalent to $\s M$, and since both are minimal systems, this would imply $\s M(c)\cong \sM$, which is a contradiction.
\end{remark}
\begin{remark}
    It is known that a cohomology algebra $A$ given by $A=\QQ(x_1, \cdots, x_s)/I$, where $I$ is an ideal generated by a regular sequence, is intrinsically formal (see \cite[Theorem 5.2]{ShigaYagita} or \cite{Sullivan_infinitesimal}). However, \cref{ex: IF in each level but not eq IF} shows that even though the levels are free CDGAs modulo a regular sequence, $\s A$ is not intrinsically formal. Thus, the aforementioned theorem does not generalize to the equivariant setting.
\end{remark}

\cref{{ex: IF in each level but not eq IF}} raises a question of whether the converse holds, namely,  whether $G/H$-orbit levels of an equivariantly intrinsically formal system of $\GA$s are intrinsically formal. We give a positive answer to this question under suitable assumptions.
\begin{definition}[Extension property]\label{defn: extension prop}
    We say a diagram of $\GA$s, $\sA$, over $\mathscr{O}_G$ satisfies extension property at a subgroup $H\leq G$ if for any DGA $B_H$ with
$
\H^\ast(B_H) \cong \H^\ast(\sA(G/H)),
$
there exists a system of $\DGA$s $\s B$ 
such that
\begin{enumerate}
\item $\s B(G/H) = B_H$,
\item $\uH^\ast(\sB) \cong \uH^\ast(\s A)$ as diagrams of graded algebras.
\end{enumerate}

\end{definition}
\begin{proposition}
Let $G = \C_{p^k}$ and let $\s A$ 
be a system of $\GA$s over $\s O_G$. Assume:
\begin{enumerate}
\item $\s A$ is intrinsically formal as a diagram, and
\item $\s A$ satisfies the extension property at every $H\leq G$,
\end{enumerate}
Then for every subgroup $H \leq G$, the $\
GA$, $\sA(G/H)$ is intrinsically formal.
\end{proposition}
\begin{proof}
Fix a subgroup $H \leq G$. Let $B_H$ be any DGA such that
$
\H^\ast(B_H) \cong \H^\ast(\sA(G/H)).
$
We must show that $\sA(G/H)$ is weakly equivalent to $ B_H$.

By the extension property, there exists a system of $\DGA$s $\s B$
such that
$
\sB(G/H) = B_H
$
and
$
\uH^\ast(\sB) \cong \uH^\ast(\sA)
$
as diagrams of graded algebras.

Since $\s A$ is intrinsically formal, it follows that there exists a weak equivalence of diagrams
$
\sA\to \cdots \leftarrow  \sB.
$
Evaluating at $G/H$, we obtain a weak equivalence
$
\sA(G/H) \to \cdots \leftarrow \sB(G/H) = B_H.
$
Since $B_H$ was arbitrary, it follows that $\sA(G/H)$ is intrinsically formal.
\end{proof}
\begin{example}
    Consider a systen of $\GA$s $\s A$ over $\s O_{\C_{p^k}}$ such that $\s A(G/H)=\QQ$ for all $H>e.$ In this case, $\s A$ satisfies the extension property at each $H\leq G.$ Indeed, let $B_H$ be a $\GA$ with the same cohomology as $\s A(G/H).$
    
   If $H\neq e$, then we define an extension
   $\s B$ by  $$\s B(G/e)=\s A(G/e)\otimes B_H,\qquad \s B(G/K)=B_H \qquad\text{for }K\neq e.$$ 
   
    On the other hand, if $H=e$, given $B_e$, we define $\s B$ by $$\s B(G/e)=B_e,\qquad \s B(G/K)=\QQ \qquad\text{for }K\neq e.$$ 
    Therefore, by the above proposition, if $\s A$ is equivariantly intrinsically formal, then the orbit level $\GA$s are intrinsically formal.
\end{example}
\vspace{1.5em}
\noindent \textbf{Acknowledgment:}
Bittu Singh gratefully acknowledges the support of the Prime Minister's Research Fellowship (PMRF ID 1302644), Government of India.

\vspace{1.5em}
\noindent \textbf{AI Declaration:}
No generative AI was used for the proofs or the computations presented in the manuscript.
\bibliographystyle{alpha}
\bibliography{reference}

    \end{document}